\documentclass[11pt,letter]{article}

\usepackage{etex}
\usepackage{titlesec}
\usepackage[top=0.75in,bottom=0.75in,left=0.75in,right=0.75in]{geometry}

\titleformat{\subsubsection}[runin]{\normalfont\large\bfseries}{\thesubsubsection}{1em}{}

\usepackage{amsmath,mathrsfs, textcomp}
\usepackage{amsthm}
\usepackage{amssymb}
\usepackage{pstricks,pst-plot,psfrag}
\usepackage[all]{xy}
\usepackage{graphicx,subfigure,xspace,bm}
\usepackage{mathtools}

\usepackage{algorithm}
\usepackage[noend]{algorithmic}
\algsetup{indent=2em}

\newtheorem{theorem}{Theorem}[section]

\newtheorem{definition}[theorem]{Definition}
\newtheorem{assumption}[theorem]{Assumption}
\newtheorem{lemma}[theorem]{Lemma}

\newtheorem{remark}[theorem]{Remark}

\newtheorem{corollary}[theorem]{Corollary}
\newtheorem{proposition}[theorem]{Proposition}  
\newcommand{\real}{{\mathbb{R}}}

\newcommand{\Tr}{\mathrm{Tr}}

\usepackage{tikz}
\usetikzlibrary{external}
\usetikzlibrary{decorations.shapes,arrows,calc,decorations.markings,shapes,decorations.pathreplacing,shapes.arrows}
\usetikzlibrary{positioning}
\tikzset{main node/.style={circle,fill=blue!20,draw,inner sep=1pt},}

\usepackage[T1]{fontenc}
\usepackage[osf]{mathpazo}

\definecolor{BBlue}{cmyk}{.98,0.10,0,.25}

\usepackage[title]{appendix}

\begin{document}
	
\title{\rule{\textwidth}{0.4mm} {\textsc{An Iterative Riccati Algorithm for Online Linear Quadratic Control}}\\
\rule{\textwidth}{0.4mm}}
\date{}
\maketitle
	
\vspace{-2cm}
\begin{flushright}
{\footnotesize {{\bf Mohammad Akbari}}\footnote{ Department of Mathematics and Statistics at Queen's University, \texttt{13mav1@queensu.ca}. 
}\hspace{1cm}}\\
{\footnotesize {{\bf Bahman Gharesifard}}\footnote{ Department of Mathematics and Statistics at Queen's University, \texttt{bahman.gharesifard@queensu.ca}. 
}\hspace{1cm}}\\
{\footnotesize {{\bf Tamas Linder}}\footnote{ Department of Mathematics and Statistics at Queen's University, \texttt{tamas.linder@queensu.ca}. 
}\hspace{1cm}}
\end{flushright}  
\sloppy

\maketitle
\begin{abstract}
An online policy learning problem of linear control systems is studied. In this problem, the control system is known and linear, and a sequence of quadratic cost functions is revealed to the controller in hindsight, and the controller updates its policy to achieve a sublinear regret, similar to online optimization. A modified online Riccati algorithm is introduced that under some boundedness assumption leads to logarithmic regret bound. In particular, the logarithmic regret for the scalar case is achieved without boundedness assumption. Our algorithm, while achieving a better regret bound, also has reduced complexity compared to earlier algorithms which rely on solving semi-definite programs at each stage.
\end{abstract}

\section{Introduction}\label{sec:intro}

 Decision making based on predictions is a cornerstone of engineering, economy, and social sciences. Examples include 
portfolio selection \cite{AA-EH-SK-RES:06,HL-CW-KZ:18}, transportation and traffic control \cite{MP-NR:01}, power engineering \cite{JZ-YL-HC:16}, manufacturing and supply chain management, show promising achievements and have received considerable attention in recent years, see, e.g.,~\cite{OA-EH-SM-OS:13,SR-GG-DB:11,NCB-GL:06}. The problem setting that we address in this paper belongs to a class of decision making problems known as \emph{online optimization}. The literature on this subject is extremely rich and its connections to many other areas of learning have been explored in recent years~\cite{NCB-GL:06,EH:16,SSS:12,EH-AA-SK-07,EH-SK:14,EG-NC-CG-YM:13,AB-YM:07}. 

Unlike the general setting of online optimization, where the decisions of the learner are solely chosen according to a time-varying cost function, in many realistic scenarios the learner's decisions involve a \emph{control system}, where the decisions not only incur a cost at the present time, but also change the system state and incurs a cost at the future.
 Examples include power supply management in the presence of time-varying energy costs due to demand fluctuations and tracking of an adversarial target. In this problem setting, decisions are assumed to be a function of the current state which is referred to as a \emph{policy}. In order to assess the performance of policies over time, a regret is defined as the difference between the accumulated costs incurred by the control actions made in hindsight using previous states and the cost incurred by the best fixed admissible policy when all the cost functions are known in advance. Similar to online optimization, the objective is to design algorithms to generate policies which make the regret function grow sublinearly. Clearly if the cost functions were available to the decision maker, the problem discussed above would reduce to the classical optimal control problem.
 
The problem setting here is similar to the one studied in~\cite{AC-AH-TK-NL-YM-KT:18} where an online version of linear quadratic Gaussian control is introduced. 
In particular, in~\cite{AC-AH-TK-NL-YM-KT:18} an online gradient descent algorithm with a fixed learning rate is proposed, where in each iteration, a projection 
onto a bounded set of positive-definite matrices is taken, which itself relies on solving a semi-definite program. Under the assumptions that the underlying system is controllable, the cost functions are bounded, and the covariance of the disturbance is positive-definite, it is proved that the regret is sublinear, and grows as $\mathcal{O}(\sqrt{T})$, where $T$ is the time horizon. 
Other closely related works are~\cite{NA-BB-EH-SK-KS:19} and~\cite{NA-EH-KS:19}, where the cost functions are assumed to be general convex and globally Lipschitz functions. In contrast to \cite{AC-AH-TK-NL-YM-KT:18}, the noise assumed in~\cite{NA-BB-EH-SK-KS:19} is adversarial, and \cite{NA-EH-KS:19} achieves a regret bound of $\mathcal{O}((\log(T))^7)$. In these works, the generated control actions, which lead to a sublinear regret bound, are linear feedbacks which rely on a finite history of the past disturbances. Similarly, in another recent work~\cite{DF-MS:20} a fixed (known) system with adversarial disturbances and fixed (known) quadratic cost functions is assumed, and a regret bound of $\mathcal{O}((\log(T))^3)$ is achieved.

Here, we point out a wider set of literature related to our work. First, we note that one can think about the underlying control system as a dynamical constraint on the optimization problem. Considering control systems as constraints is also classical in the context of \emph{model predictive control}~\cite{CEG-DMP-MM:89}. Although we tackle dynamic constraints in this work, we should emphasize that online optimization problems with static constraints, known only in hindsight, also play a key role in various settings and have generated interest in recent years~\cite{HY-MN-XW:17,MJN-HY:17,RJ-JH-CA:16}.


Our work is also related to the framework of Markov decision processes (MDPs), where the system transition to the next state is defined through a probability distribution. Moreover, a reward is given to the decision maker for each action at each state. This framework is classical in \emph{reinforcement learning}, where the objective is to learn the optimal policy which yields the maximum reward~\cite{RS-AGB:18}. It is also worth pointing out that there is another key role that regret minimization has played recently, bringing learning and control theory together, in the context of robust control, adaptive control, and system identification. Here, the regret enters through the lack of perfect knowledge of the model, and research efforts focus on generating algorithms for updating models in a data-driven fashion~\cite{YY-ZG-HX-DD-YY-DCW:19,AK-CK:17}. Finally, our setting is also related to online optimization in dynamic environments~\cite{ECH-RMW:13}, where the decisions are constrained in dynamics chosen by the environment. However, the objective of~\cite{ECH-RMW:13} is to study the impact of model mismatch on the overall regret, whereas in this paper the decisions are input to a control system, which impacts the way the decisions affect the future outcomes through its dynamics. 

{\bf Contributions.} We consider the problem of online linear quadratic Gaussian optimal control,  where the control system is linear and known and the cost function is quadratic and time-varying and only becomes available in hindsight. In contrast to~\cite{AC-AH-TK-NL-YM-KT:18}, where an online algorithm using semi-definite programming update is designed to generate the control policies, we employ a control-theoretic approach and introduce an online version of a classical iterative Riccati update, known as the Newton-Hewer~\cite{GH:71} update. Using this update, which is less known than the classical Riccati difference equation~\cite{PEC-DQM:70}, is key in developing our algorithm. This algorithm, which employs only a few matrix addition and multiplication operations in each time step, has reduced complexity compared to the one using semi-definite programming in each time step and is easier to implement. Our main result is a $\mathcal{O}(\log{T})$ regret bound for the online linear quadratic Gaussian optimal control problem, improving the $\mathcal{O}(\sqrt{T})$ bound of~\cite{AC-AH-TK-NL-YM-KT:18} and the $\mathcal{O}((\log(T))^7)$ bound of~\cite{NA-EH-KS:19} for time horizon $T$, under some boundedness assumption. Indeed, the technical part of our result relies on characterizing the interplay between a notion of stability for the sequence of control policies and boundedness of the solutions of the proposed Riccati update. The latter boundedness property, which follows for the Riccati difference equation from monotonicity with respect to the underlying parameters, cannot be obtained using monotonicity; in fact, the Newton-Hewer updates can fail to be monotone in this setting~\cite{MA-BG-TL:20}.
This being said, for the scalar case, we are able to prove that boundedness can in fact be verified, yielding the stronger result that initializing the control policy to be stable is enough to guarantee boundedness of the solutions of the proposed online Riccati update.

{\bf Notation.} We let $\real$ denote the set of real numbers and $\real^{n \times m}$ denote the set of $n \times m$ real matrices. We use lowercase letters for vectors and uppercase letters for matrices. We denote by $\|\cdot\|$ the Euclidean norm on vectors and its corresponding operator norm on real matrices. We denote by $A^\top$ the transpose of matrix $A$. Thus $\|A\|=\sigma_{\max}(A)=\sqrt{\lambda_{\max}(A^\top A)}$, where $\sigma_{\max}(A)$ is the largest singular value of $A$ and $\lambda_{\max}(A^\top A)$ is the largest eigenvalue of $A^\top A$. Trace of matrix $A$ is denoted by $\Tr(A)$. If $A$ is an $n\times n$ real matrix with eigenvalues $\lambda_1,\ldots, \lambda_n$, then the spectral radius of $\rho(A)$ of $A$ is $\rho(A)=\max\{|\lambda_1|,\ldots,|\lambda_n|\}$. We use $A\succeq B$ to indicate that $A-B$ is positive semi-definite.

\section{Problem Formulation}

We start by describing the problem of online optimization for the class of linear control systems with quadratic cost. Let us recall this setting. 

\subsection{Discrete-Time Linear Quadratic Gaussian Control}
The discrete-time linear quadratic Gaussian (LQG) control problem is defined as follows, see for instance~\cite{ST:02}: Let $x_t\in\real^{n}$ and $u_t\in\real^{m}$ be the control state and the control action at time $t$, respectively, with initial state $x_1$. The system dynamics are given by 
\begin{align}\label{contsys}
	x_{t+1}=Ax_t+Bu_t+w_t, \qquad t\geq 1
\end{align}
where $A\in\real^{n\times n}$, $B\in\real^{n\times m}$, and $\{w_t\}_{t\geq 1}$ are i.i.d.~Gaussian noise vectors with zero mean and covariance $W\in\real^{n\times n}$ ($w_t \sim \mathcal{N}(0,W)$). It is assumed that the initial value is Gaussian $x_1\sim\mathcal{N}(m,X_1)$ and is independent of the noise sequence $\{w_t\}_{t\geq 1}$. 
The cost incurred in each time step $t$ is a quadratic function of the state and control action given by $x_t^\top Q_t x_t + u_t^\top R_t u_t$, where $Q_t\in \real^{n\times n}$ and $R_t\in\real^{m\times m}$ are positive-definite matrices. The total cost after $T$ time steps is given by
\[
J_T(x_1, u_1,\ldots,u_T)=\mathbb{E}\Big[x_T^\top Q_T x_T+\sum_{t=1}^{T-1}\big(x_t^\top Q_t x_t+u_t^\top R_t u_t\big)\Big].
\]
We consider controllers of the form $u_t=\pi_t (x_t)$, where the function $\pi_t:\real^n \to \real^m$ is called a policy. This assumption does not restrict generality, as the optimal policy will provably be of this form~\cite{ST:02}. It is well-known that under the assumption that the control system is stabilizable, and the cost matrices $Q_t$ and $R_t$ are positive-definite, the optimal policy is a stable linear feedback of the state, which will be described next. 

\subsection{Discrete Algebraic Riccati Equation}
In the classical LQG problem, where all the cost functions are known, the optimal policy can be obtained by dynamic programming, and is a linear function of the state. In particular, $u_t=-K_t x_t$, where $K_t$ is given by the equation
\[
K_t=(B^\top P_{t+1}B+R_{t})^{-1}B^\top P_{t+1}A,
\]
and $P_{t+1}$ is a sequence of positive-definite matrices obtained iteratively, backwards in time, from the dynamic Riccati equation:
\begin{equation}\label{fdare}
P_{t}=A^\top P_{t+1} A-A^\top P_{t+1}B(B^\top P_{t+1}B+R_t)^{-1}B^\top P_{t+1} A +Q_t
\end{equation}
with the terminal condition $P_T=Q_T$. 

For the infinite-horizon problem with the assumption that $Q_t=Q$ and $R_t=R$ are fixed, and under the  assumptions that
\begin{enumerate}
\item $R$ is positive-definite
\item $(A,B)$ is stabilizable, i.e., there exists a linear policy $\pi(x)=-Kx$ such that the closed-loop system $x_{t+1}=(A-BK)x_t$ is asymptotically stable: $\rho(A-BK)<1$,
\item $(A,C)$ is detectable where $Q=C^\top C$, [i.e., if $u_t\rightarrow 0$ and $Cx_t\rightarrow 0$ then, $x_t\rightarrow 0$],
\end{enumerate}
it is well-known that the optimal policy is unique, time invariant, and is a linear function of the state~\cite{DPB:18}, i.e., $u_t=-K^\star x_t$. Here $K^\star$ is given by
\begin{align}
K^\star&=(B^\top P^\star B+R)^{-1}B^\top P^\star A,
\end{align}
where $P^\star$ satisfies the discrete algebraic Riccati equation (DARE):
\begin{align}\label{dare}
P^\star&=A^\top P^\star A-A^\top P^\star B(B^\top P^\star B+R)^{-1}B^\top P^\star A +Q.
\end{align} 
Moreover, $P_t$ given by \eqref{fdare} converges to $P^\star$ as $t\to\infty$~\cite{ST:02}.
By using the policy $K^\star$, 
we have that $x_{t+1}=(A-BK^\star)x_t+w_t$. The optimal policy $K^\star$ 
is guaranteed to be stable i.e. $\rho(A-BK^\star)<1$. Here, $x_t$ converges to a stationary distribution, i.e., $x_t$ converges weakly to a random variable $x$ which has the same distribution as $(A-BK^\star)x+w_t$, so that we have $\mathbb{E}[x]=\mathbb{E}[(A-BK^\star)x+w_t]$, which implies $\mathbb{E}[x]=0$, and the covariance matrix $X=\mathbb{E}[x x^\top]$ satisfies $X=(A-BK^\star)X(A-BK^\star)^\top+W$, see e.g.,~\cite{AC-AH-TK-NL-YM-KT:18}. 

\subsection{Problem Setting}
We now define the problem we study in this work, following~\cite{AC-AH-TK-NL-YM-KT:18}. In \emph{online linear quadratic control}, the sequence of cost matrices $\{Q_t\}_{t\geq 1}$ and $\{R_t\}_{t\geq 1}$ are not known in advance and $Q_t$ and $R_t$ are only revealed after choosing the control action $u_t$. Since it is not possible to find the optimal policy before observing the whole sequence of cost matrices $\{Q_t\}_{t\geq 1}$ and $\{R_t\}_{t\geq 1}$, the decision maker faces a \emph{regret}. Here, we assume that the control system $(A,B)$ is stabilizable, and the cost matrices $Q_t$ and $R_t$ are positive-definite and uniformly bounded over $t\geq 1$. As the optimal policy for the system with these assumptions is given by a stable linear feedback, we use the set of stable linear feedback functions as the set of admissible policies. 

Let $x_t\in\real^n$ and $u_t\in\real^m$ be the control state and controller action at time $t\geq 1$. The controller uses a linear feedback policy $u_t=-K_t x_t$ and commits to this action after observing $x_t$. Then the controller receives the positive-definite matrices $Q_t\in \real^{n\times n}$ and $R_t\in\real^{m\times m}$, and suffers the cost 
\begin{equation}
J_t(K_t)=\mathbb{E}\Big[x_t^\top Q_t x_t + u_t^\top R_t u_t\Big].
\end{equation}
The objective is to design an algorithm 
to generate a sequence of policies $\{K_t\}_{t\geq 1}$ such that the regret function, which is defined as 
\begin{equation}\label{eq:regret-def}
R(T)=\sum_{t=1}^T J_t(K_t)-\min_{K\in\mathcal{K}}\sum_{t=1}^T J_t(K),
\end{equation}
where $\mathcal{K}$ is the set of stable policies, grows sublinearly in $T$. In other words, the average regret over time converges to zero. Before stating our main results, we provide a brief review of the iterative Riccati updates that we employ to design our main algorithm. 

\subsection{Iterative Methods for Solving the Discrete Algebraic Riccati Equation}\label{background}
Several methods for solving DARE exist in the literature, including iterative methods  \cite{PEC-DQM:70}, algebraic methods  \cite{LR-PL:95}, and semi-definite programming  \cite{VB-LV:03}. Our work is based on iterative methods, and in particular, two techniques that we review here. The first is given in \cite{PEC-DQM:70},
where one runs the recursion 
\[
P_{t+1}=A^\top P_{t} A-A^\top P_{t}B(B^\top P_{t}B+R)^{-1}B^\top P_{t} A +Q.
\] 
It is shown that under the assumption that $(A,B)$ is stabilizable and $(A,C)$ is detectable, where $Q=C^\top C$, the sequence $\{P_t\}$ converges to the unique solution of DARE.

A second approach, studied by \cite{GH:71}, uses the following idea: Let $P_t$ be the solution of the equation
\begin{align}\label{Halg}
P_{t}=(A-BK_t)^\top P_t (A-BK_t)+K_t^\top R K_t +Q,
\end{align}
where
\begin{align*}
K_t=(B^\top P_{t-1}B+R)^{-1}B^\top P_{t-1}A,
\end{align*}
starting from a stable policy $K_1$. Then under the assumption that $(A,B)$ is stabilizable and $(A,C)$ is detectable, where $Q=C^\top C$, the sequence $\{P_t\}$ converges to the solution of DARE and the rate of convergence is quadratic, i.e.,
\begin{equation*}
\|P_t-P^\star\|\leq C\|P_{t-1}-P^\star\|^2
\end{equation*}
where $C>0$ is a constant. In what follows, we modify this algorithm and use it for the online linear quadratic Gaussian problem. 
We present our algorithm after reviewing some salient properties of stable policies. Similar to~\cite{AC-AH-TK-NL-YM-KT:18}, we use the notion of \emph{strong stability}, which allows us to analyze the rate of convergence of the state covariance matrices under our proposed algorithm.

\section{Strong Stability}
A key property that we require before introducing our algorithm is the notion of strong stability and sequential strong stability which are similar to the ones in~\cite{AC-AH-TK-NL-YM-KT:18}. The notion of strong stability is defined as follows.
\begin{definition}\label{defstab}  
A policy $K$ is called stable if $\rho(A-BK)<1$. A policy $K$ is $(\kappa, \gamma)$-strongly stable (for $\kappa>0$ and $0<\gamma\leq 1)$ if $\|K\|\leq \kappa$, and there exist matrices $L$ and $H$ such that $A-BK=HLH^{-1}$, with $\|L\|\leq 1-\gamma$ and $\|H\|\|H^{-1}\|\le \kappa$.
\end{definition}

Note that every $(\kappa,\gamma)$-strongly stable policy $K$ is stable, since the matrices $A-BK$ and $L$ are similar and hence $\rho(A-BK)=\rho(L)\leq (1-\gamma)$. Lemma~\ref{lemma:Co1} shows that every stable policy is $(\kappa,\gamma)$-strongly stable for some $\kappa>0$ and $0<\gamma\leq 1$.

\begin{lemma}\cite[Lemma B.1.]{AC-AH-TK-NL-YM-KT:18}\label{lemma:Co1}
Suppose that for a linear system defined by $A$, $B$, a policy $K$ is stable. Then there are parameters $\kappa>0$, $0<\gamma\leq 1$ for which it is $(\kappa,\gamma)$-strongly stable.
\end{lemma}

We refer the reader to \cite[Lemma B.1]{AC-AH-TK-NL-YM-KT:18} for a proof of this lemma. 

Under the assumption of $(\kappa, \gamma)$-strong stability of policy $K$, the state covariance matrices $X_t=\mathbb{E}[x_t x_t^\top]$ converge exponentially to a steady-state covariance matrix $\widehat{X}$, which satisfies \[\widehat{X}=(A-BK)\widehat{X}(A-BK)^\top +W.\]
Lemma~\ref{stabconv} provides the details. 

\begin{lemma}\cite[Lemma 3.2]{AC-AH-TK-NL-YM-KT:18}\label{stabconv}
Let the pair $(A,B)$ be stabilizable, 
and assume the controller uses a fixed $(\kappa,\gamma)$-strongly stable policy $K$, i.e., for $t\geq 1$, we have $u_t=-Kx_t$. Let $X_t$ be the covariance matrix of $x_t$. 
Then the sequence $\{X_t\}_{t\geq 1}$ 
converges to the steady-state covariance matrix $\widehat{X}$, and in particular, for any $t\geq 1$,
\[
\|X_{t+1}-\widehat{X}\|\leq \kappa^2 e^{-2\gamma t}\|X_1-\widehat{X}\|.
\]
\end{lemma}
We refer the reader to \cite[Lemma 3.2]{AC-AH-TK-NL-YM-KT:18} for a proof.

In order to obtain a similar result for the change of the state covariance matrices using a sequence of different $(\kappa,\gamma)$-strongly stable policies $\{K_t\}_{t\geq 1} $, we need to define a notion of sequential strong stability, which is presented next.

\begin{definition}\label{defstrstab}
A sequence of policies $\{K_t\}_{t\geq 1}$ is sequentially $(\kappa,\gamma)$-strongly stable, for $\kappa>0$ and $0<\gamma \leq 1$, if there exist sequences of matrices $\{H_t\}_{t\geq 1}$ and $\{L_t\}_{t\geq 1}$ such that 
\[
A-BK_t=H_tL_tH_t^{-1}
\] 
for all $t\geq 1$, with the following properties:
\begin{itemize}
\item $\|L_t\|\leq 1-\gamma$ and $\|K_t\|\leq \kappa$;

\item $\|H_t\|\leq \beta$ and $\|H_t^{-1}\|\leq 1/\alpha$ with $\kappa=\beta/\alpha$ and $\alpha>0$ and $\beta >0$;

\item $\|H_{t+1}^{-1}H_t\|\leq 1+\gamma$.
\end{itemize}
\end{definition}
The importance of this notion of stability is demonstrated in Lemma~\ref{lem3}.

\begin{lemma}\label{lem3}
Let the pair $(A,B)$ be stabilizable, and suppose that the controller uses $u_t=-K_t x_t$ for $ t\geq 1 $ and where $\{K_t\}_{t\geq 1} $ is sequentially $(\kappa,\gamma)$-strongly stable with $\kappa>0$ and $0<\gamma \leq 1$. For each $K_t$, let $\widehat{X}_t$ be the corresponding steady-state covariance matrix, i.e., $\widehat{X}_t$ satisfies $\widehat{X}_t= (A-BK_t)\widehat{X}_t(A-BK_t)^\top +W$ and assume that $\|\widehat{X}_{t+1}-\widehat{X}_{t}\|\leq \eta_t$ with $\eta_t>0$, for all $t\geq 1$.
Let ${X}_t$ be the corresponding state covariance matrix at time 
 $t$, starting from some initial ${X}_1\succeq 0$. 
Then for $t\geq 1$,
\[
\|X_{t+1}-\widehat{X}_{t+1}\|\leq \kappa^2e^{-2\gamma^2 t}\|X_{1}-\widehat{X}_{1}\|+\kappa^2\sum_{s=0}^{t-1}e^{-2\gamma^2 s}\eta_{t-s}.
\]
\end{lemma}
The proof is similar to \cite[Lemma 3.5]{AC-AH-TK-NL-YM-KT:18}, but we include it for completeness. 
\begin{proof}
By definition, for all $t\geq 1$, 
we have that
\begin{align*}
X_{t+1}=&(A-BK_{t})X_{t}(A-BK_{t})^\top + W,\\
\widehat{X}_{t}=&(A-BK_{t})\widehat{X}_{t}(A-BK_{t})^\top +W.
\end{align*}
Subtracting the equations, substituting $A-BK_{t}=H_t L_t H_t^{-1}$ and rearranging yields
\[
H_{t}^{-1}(X_{t+1}-\widehat{X}_{t})(H_{t}^{-1})^\top=L_{t} H_{t}^{-1}(X_{t}-\widehat{X}_{t})(H_{t}^{-1})^\top L_{t}^\top.
\]
Let $\Delta_{t}=H_t^{-1}(X_t-\widehat{X}_t)(H_t^{-1})^\top$ for all $t\geq 1$. Then the above can be written as
\begin{align*}
\Delta_{t+1}=&(H_{t+1}^{-1}H_{t} L_{t})\Delta_{t} (H_{t+1}^{-1}H_{t} L_{t})^\top\\
&+(H_{t+1}^{-1})(\widehat{X}_{t}-\widehat{X}_{t+1})(H_{t+1}^{-1})^\top.
\end{align*}
Taking the norms yeilds
\begin{align*}
\|\Delta_{t+1}\|\leq& \|L_{t}\|^2\|H_{t+1}^{-1}H_{t}\|^2\|\Delta_{t}\|+\|H_{t+1}^{-1}\|^2\|\widehat{X}_{t}-\widehat{X}_{t+1}\|\\
\leq&(1-\gamma)^2(1+\gamma)^2\|\Delta_{t}\|+\frac{\eta_{t}}{\alpha^2}\\
\leq&(1-\gamma^2)^2\|\Delta_{t}\|+\frac{\eta_{t}}{\alpha^2},
\end{align*}
and by unfolding the recursion, we obtain
\begin{align*}
\|\Delta_{t+1}\|\leq &(1-\gamma^2)^{2t}\|\Delta_1\|+\frac{1}{\alpha^2}\sum_{s=0}^{t-1}(1-\gamma^2)^{2s}\eta_{t-s}\\
\leq & e^{-2\gamma^2 t}\|\Delta_1\|+\frac{1}{\alpha^2}\sum_{s=0}^{t-1}e^{-2\gamma^2 s}\eta_{t-s}.
\end{align*}
Using $X_t-\widehat{X}_t=H_t \Delta_t H_t^\top$ now, we have that
\begin{align*}
\|X_{t+1}-\widehat{X}_{t+1}\|\leq &e^{-2\gamma^2 t}\|\Delta_1\|\|H_{t+1}\|^2+\frac{\|H_{t+1}\|^2}{\alpha^2}\sum_{s=0}^{t-1}e^{-2\gamma^2 s}\eta_{t-s}\\
\leq &\kappa^2e^{-2\gamma^2 t}\|X_{1}-\widehat{X}_{1}\|+\kappa^2\sum_{s=0}^{t-1}e^{-2\gamma^2 s}\eta_{t-s},
\end{align*}
which concludes the proof.
\qed
\end{proof}

We now proceed with some key results that we later use to ensure strong stability for the sequence of policies generated. 
Suppose that a sequence of positive-definite matrices $P_t$ is generated recursively as
\begin{equation}\label{eq31}
P_t=(A-BK_t)^\top P_t(A-BK_t)+\bar{Q}_t+K_t^\top \bar{R}_{t}K_t,
\end{equation}
where 
\begin{equation}\label{eq32}
K_{t+1}=(B^\top P_{t}B+\bar{R}_{t})^{-1}B^\top P_{t}A
\end{equation}
and 
where $\bar{R}_{t}\in \real^{m\times m}$ and $\bar{Q}_t\in\real^{n\times n}$ are given positive-definite matrices for all $t\geq 1$, and $K_1$ is an initial stable policy. The reason for this update will become clear as part of our algorithm in Section~\ref{RicAlg}. The key point we wish to make here is that under the assumption of uniform boundedness of the matrix sequence $\{P_t\}_{t\geq 1}$, and the stability of matrix $K_t$, for all $t\geq 1$, 
the sequence $\{K_t\}_{t\geq 1}$ is uniformly $(\kappa,\gamma)$-strongly stable, with appropriate choices of $ \kappa $ and $ \gamma $. 
\begin{proposition}\label{str-stability}
Assume that for $t\geq 1$, $Q_t, R_t\succeq \mu I$ and $P_t\preceq\nu I$, where $\mu , \nu>0$ and $\{P_t\}_{t\geq 1}$ is the sequence of matrices obtained as the solution of~\eqref{eq31}, and assume that the policy $K_t$ given by~\eqref{eq32} is stable for all $t\geq 1$. Define $\bar{\kappa}=\sqrt{\frac{\nu}{\mu}}$. Then the sequence $\{K_t\}_{t\geq 1}$ is uniformly $(\bar{\kappa},1/2\bar{\kappa}^2)$-strongly stable.
\end{proposition}

\begin{proof}
By the assumption of stability and since $Q_t\succeq \mu I$, we have that
\begin{align}
P_t=&(A-BK)^\top P_t(A-BK)+\bar{Q}_t+K^\top \bar{R}_{t}K\nonumber\\
\succeq & (A-BK)^\top P_t(A-BK)+\mu I,\label{eq1}
\end{align}
where we have used the positive-definiteness of $K^\top \bar{R}_{t}K$. In particular, this means that $P_t\succeq \mu I$ for all $t$. On the other hand, assuming $P_t\preceq \nu I$, we have
\begin{equation}\label{eq2}
\mu I \preceq P_t \preceq \nu I.
\end{equation}
Given that $P_t$ is positive-definite and nonsingular, we can define $L_t=P_t^{1/2}(A-BK)P_t^{-1/2}$. Multiplying \eqref{eq1} by $P_t^{-1/2}$ from both sides, we obtain $I\succeq L_t^\top L_t+\mu P_t^{-1}\succeq L_t^\top L_t +\bar{\kappa}^{-2} I$. Thus $L_t^\top L_t\preceq (1-\bar{\kappa}^{-2})I$, so $\|L_t\|\leq \sqrt{1-\bar{\kappa}^{-2}}\leq 1-\bar{\kappa}^{-2}/2$. Also, using \eqref{eq2} we have that 
\[
\|P_t^{1/2}\|\|P_t^{-1/2}\|\leq \bar{\kappa},
\]
which finishes the proof.
\qed\end{proof}

We now present a second useful result, where we show that under the additional property that the rate of changes of sequence $P_t$ is small (which we will be able to establish for our proposed algorithm, see Lemma~\ref{lem7}), one can obtain that the sequence $\{K_t\}_{t\geq 1}$ is sequentially strongly stable. 

\begin{proposition}\label{lem6}
Assume that for $t\geq 1$, $Q_t, R_t\succeq \mu I$ and $P_t\preceq\nu I$, where $\mu, \nu>0$ and $\{P_t\}_{t\geq 1}$ is the sequence of matrices obtained as the solution of~\eqref{eq31}, and assume that the policy $K_t$ given by~\eqref{eq32} is stable for $t\geq 1$. Let $\bar{\kappa}=\sqrt{\frac{\nu}{\mu}}$, 
and suppose that $\|P_{t+1}-P_t\|\leq \eta$ for $t\geq 1$ for some $\eta\leq \mu/\bar{\kappa}^2$. Then the sequence $\{K_t\}_{t\geq 1}$ is sequentially $(\bar{\kappa},1/2\bar{\kappa}^2)$-strongly stable.
\end{proposition}

\begin{proof}
Proceeding as in the proof of Proposition \ref{str-stability}, one can show that the matrix $L_t=P_t^{1/2}(A-BK_t)P_t^{-1/2}$ satisfies $\|L_t\|\leq 1-1/2\bar{\kappa}^2$ with $\|P_t^{1/2}\|\leq \sqrt{\nu} $ and $\|P_t^{-1/2}\|\leq 1/\sqrt{\mu}$. To establish the sequential strong stability stated by Definition~\ref{defstrstab} 
it thus suffices to show that $\|P_{t+1}^{-1/2}P_t^{1/2}\|\leq 1+1/2\bar{\kappa}^2$ for $t\geq 1$. To this end, observe that $\|P_{t+1}-P_t\|\leq \eta$, and that 
\begin{align*}
\|P_{t+1}^{-1/2}P_t^{1/2}\|^2&=\|P_{t+1}^{-1/2}P_t P_{t+1}^{-1/2}\|\\
&\leq \|P_{t+1}^{-1/2}P_{t+1}P_{t+1}^{-1/2}\|+\|P_{t+1}^{-1/2}(P_{t+1}-P_t)P_{t+1}^{-1/2}\|\\
&\leq 1+\|P_{t+1}^{-1/2}\|^2\|P_{t+1}-P_t\|\\
&\leq 1+\frac{\eta}{\mu},
\end{align*}
where the second inequality follow by the sub-multiplicative of matrix operator norm. 
Hence, since $\eta\leq \mu/\bar{\kappa}^2$, then $\|P_{t+1}^{-1/2}P_t^{1/2}\|\leq \sqrt{1+1/\bar{\kappa}^2}\leq 1+1/2\bar{\kappa}^2$ as required.
\qed\end{proof}

The above results rely on uniform boundedness of the sequence $\{P_t\}_{t\geq 1}$, which we assume throughout the paper. However, we can show that stability of $K_1$ is enough to guarantee this property in the scalar case, see Proposition~\ref{prop1} in the Appendix. Based on our extensive simulation studies, one of which is shown in Example~\ref{example:A}, we believe that this property should hold only by assuming stability of $ K_1 $ for the general case. One of the main reason for the difficulty of establishing this result is the lack of monotonicity of the evolutions of the 
Newton-Hewer dynamics with respect to the underlying system parameter, a sharp contrast with the Riccati difference updates~\cite{PEC-DQM:70}, which we have recently reported in~\cite{MA-BG-TL:20}. In this sense, the proof of Proposition~\ref{prop1} for the scalar case establishes the boundedness property of the sequence $\{P_t\}_{t\geq 1}$ without relying on monotonicity. 
We have outlined further details in Remark~\ref{remark:A}.

\section{The Online Riccati Algorithm}\label{RicAlg}

We outline our main algorithm in this section. 
Our assumptions are as follows: 

\begin{assumption}\label{assump} Throughout we assume that
\begin{itemize}
\item The pair $(A,B)$ is stabilizable.
\item The cost matrices $Q_t$ and $R_t$ are positive-definite and $\mu I \preceq Q_t$, $\mu I \preceq R_t $, and $\Tr(Q_t)\leq \sigma$, $\Tr(R_t)\leq \sigma$, for some $\sigma>\mu>0$ for all $t\geq 1$.
\item For the noise covariance matrix $W$ we have that $\omega=\Tr(W)<\infty$. 
\end{itemize}
\end{assumption}
\begin{algorithm}
	\caption{\textbf{Online Riccati Update}}
	\label{alg}
\begin{algorithmic}[1]
  \REQUIRE The system matrices $A$ and $B$, initial state $x_1$, time horizon $T$, parameters $\nu,\mu, \kappa=\sqrt{\nu/\mu}, \gamma=1/(2\kappa^2), \sigma$ 
  \ENSURE A sequence of stable policies $\{K_t\}_{t=1}^T$
  \STATE \textbf{Initialize} $K_1$ to be stable
  \FOR{each $t=1,2,\cdots,T$}
  	\STATE receive $x_t$
  	\STATE use controller $u_t=-K_t x_t$ and receive $Q_t$ and $R_t$
  	\STATE update $\bar{R}_t=\frac{t-1}{t}\bar{R}_{t-1}+\frac{1}{t}R_t$, $\bar{Q}_t=\frac{t-1}{t}		\bar{Q}_{t-1}+\frac{1}{t}Q_t$
  	\STATE update $P_t$ as the solution of \label{eq5}
\begin{equation*}
P_t=(A-BK_t)^\top P_t(A-BK_t)+\bar{Q}_t+K_t^\top \bar{R}_{t}K_t
\end{equation*}
	\STATE \textbf{Reset:} \IF{$t=t^\star:=\big\lceil{\frac{4\kappa^3\|B\|}{\gamma\mu}(2\sigma\kappa+\frac{2\kappa^3\|B\|\sigma(1+\kappa^2)}{\gamma})+1}\big\rceil$}
	\STATE Initialize $\ell=0$, $\widehat{P}_{0}=P_{t^\star}$, and $\widehat{K}_{0}=K_{t^\star}$
	\WHILE{$\|\widehat{P}_{\ell}-\widehat{P}_{\ell-1}\|>(\frac{2\sigma}{\|B\|}+\frac{4\kappa^2\sigma(1+\kappa^2)}{\gamma})/t^\star$}
	\STATE $\ell\leftarrow \ell+1$
	\STATE $\widehat{K}_{\ell}=(B^\top \widehat{P}_{\ell-1}B+\bar{R}_{t^\star})^{-1}B^\top \widehat{P}_{\ell-1}A$
	\STATE $\widehat{P}_{\ell}$ satisfies $\widehat{P}_{\ell}=(A-B\widehat{K}_{\ell})^\top \widehat{P}_{\ell}(A-B\widehat{K}_{\ell})+\bar{Q}_{t^\star}+\widehat{K}_{\ell}^\top \bar{R}_{t^\star}\widehat{K}_{\ell}$
	\ENDWHILE
	\RETURN $P_{t^\star}=\widehat{P}_{\ell}$
	\ENDIF
	\RETURN $K_{t+1}=(B^\top P_{t}B+\bar{R}_{t})^{-1}B^\top P_{t}A$
  \ENDFOR
\end{algorithmic}
\end{algorithm}

A formal description is given in Algorithm~\ref{alg}. We provide an informal description. We start from a stable policy $K_1$; the existence of $K_1$ is provided by the assumption of stabilizability of the control system. At each time step $ t \geq 1 $, 
the controller uses the policy $u_t=-K_t x_t$ after observing $x_t$, then the cost matrices $Q_t$ and $R_t$ are revealed, and the controller updates $P_t$ and $K_t$ using the average of the history of $Q_t$s and $R_t$s through~\eqref{Halg}. There is a technical step in our algorithm, which we call the ``reset'' step and describe in detail later in the proof; this step allows us to show that using these updates the change of the norm of the policies is $\mathcal{O}(1/t)$, 
and this gives a regret bound $\mathcal{O}(\log(T))$. Before we state the algorithm, we need to elaborate on the parameters used.

\begin{remark}[Parameters used in Algorithm~\ref{alg}]
Our algorithm naturally uses parameters $\mu$ and $\sigma$, stated in Assumption~\ref{assump}. For the reset step, we also need (an estimate on) the strong stability parameters $\kappa$ and $\gamma$, which are defined in Algorithm~\ref{alg}. Proposition~\ref{str-stability} plays a key role in that regard, as it states that as long as we can estimate a uniform bound on the sequence $P_t$, we can obtain these parameters. In the scalar case, we know this uniform bound by Proposition~\ref{prop1}; in other cases, given that the parameters are not needed in the early steps of the algorithm, one can envision that we can run our algorithm with a large estimate on this bound and adjust it if necessary. Extending Proposition~\ref{prop1} to vector cases, which is an avenue of our current research, will remove this restriction all together.
\end{remark}

\section{Main Results}\label{Main}

We are now in a position to state our main contribution, providing a logarithmic bound for the regret~\eqref{eq:regret-def}. We have opted not to explicitly display the bound as part of the statement; this can be found in~\eqref{eq:thm}.

\begin{theorem}\label{theorem:main}
Suppose that the tuple $(A,B,\{Q_t\}_{t=1}^T,\{R_t\}_{t=1}^T, W)$ satisfies Assumption~\ref{assump}. Suppose that the matrices $P_t$ generated by Algorithm~\ref{alg} are uniformly bounded. 
Then we have that
\begin{equation*}
\mathcal{R}(T)=\mathcal{O}(\log(T)).
\end{equation*}
\end{theorem}


The rest of this section is devoted to proving Theorem~\ref{theorem:main}. The proof is quite involved, and for this reason we find it useful to provide a brief description to help the reader navigate through the proof. Our first technical result Lemma~\ref{stability} shows that Algorithm~\ref{alg}, as long as it is initialized at an stable policy, iteratively produces stable polices. This step is analogous to the classical result of~\cite{GH:71} for the case where the cost objective matrices $Q_t$ and $R_t$ are fixed. Recall that, by Proposition~\ref{str-stability}, stability of policies $K_t$ is required to establish strong stability. %
A technical part of this proof demonstrates the reason why we need the reset step of the algorithm to ensure that the sequence of policies $ \{P_{t+1}-P_t\} $ decay as $ m/t $, for some $ m> 0 $. Using this and by rewriting the regret using trace products, we establish a set of bounds in Lemmas~\ref{lem8},~\ref{lem9}, and \ref{lem10}~which eventually yield the result. 

\subsection*{Proof of Theorem~\ref{theorem:main}:} 

We first provide a straightforward reformulation of the regret function. For matrices $A$ and $B$ of appropriate size, let $A\bullet B=\Tr(A^\top B)$. Then

\begin{align}
R(T)=&\sum_{t=1}^T \mathbb{E}\Big[x_t^\top Q_t x_t + u_t^\top R_t u_t\Big]-\sum_{t=1}^T \mathbb{E}\Big[{x_t^\dagger}^\top Q_t x_t^\dagger +{x_t^\dagger}^\top K^{\dagger\top} R_t K^{\dagger}x_t^\dagger\Big]\nonumber\\
=&\sum_{t=1}^T (Q_t+K_t^\top R_t K_t)\bullet X_t-\sum_{t=1}^T (Q_t+K^{\dagger\top} R_t K^\dagger)\bullet X^\dagger_t\nonumber\\
\end{align}
As a result, we have that 
\begin{align}
R(T)=&\sum_{t=1}^T (Q_t+K_t^\top R_t K_t)\bullet (X_t-\widehat{X}_t)\label{eqx1}\\
&+\sum_{t=1}^T (Q_t+K_t^\top R_t K_t)\bullet \widehat{X}_t-\sum_{t=1}^T (Q_t+K^{\star\top} R_t K^\star) \bullet\widehat{X}^\star\label{eqx2}\\
&+\sum_{t=1}^T (Q_t+K^{\star\top} R_t K^\star) \bullet\widehat{X}^\star-\sum_{t=1}^T (Q_t+K^{\dagger\top} R_t K^\dagger) \bullet\widehat{X}^\dagger\label{eqx4}\\
&+\sum_{t=1}^T (Q_t+K^{\dagger\top} R_t K^\dagger)\bullet (\widehat{X}^\dagger-{X}_t^\dagger)\label{eqx3},
\end{align}
where $K^\dagger$ is the fixed optimal policy for the system $(A, B, \{Q_t\}_{t=1}^T, \{R_t\}_{t=1}^T, W)$, $X_t=\mathbb{E}[x_t x_t^\top]$ is the covariance matrix of $x_t$ when the system follows policies $K_t$ generated by Algorithm \ref{alg}, 
$\widehat{X}_t$ is the steady-state covariance matrix using the policy $K_t$, i.e. $\widehat{X}_t$ satisfies 
\[
\widehat{X}_t=(A-BK_t)\widehat{X}_t(A-BK_t)^\top + W,
\] 
and 
\[
X^\dagger_t=\mathbb{E}[x_t^\dagger {x_t^\dagger}^\top]
\]
is the covariance matrix of the state $x_t^\dagger$ at time $t$ when the system uses policy $K^\dagger$ at each time $t$; similarly, 
$\widehat{X}^\dagger$ is the steady-state covariance matrix using the policy $K^\dagger$, i.e., $\widehat{X}^\dagger$ satisfies 
\begin{equation}\label{eqre1}
\widehat{X}^\dagger=(A-BK^\dagger)\widehat{X}^\dagger(A-BK^\dagger)^\top + W.
\end{equation}
$K^\star$ is the solution to DARE and $\widehat{X}^\star$ is the steady-state covariance matrix using policy $K^\star$. From now on, we use the notation $A_t=A-BK_t$ to simplify the presentation.

Note that by the following computation we show that \eqref{eqx4} is negative. Since $\widehat{X}^\star$ and $\widehat{X}^\dagger$ are fixed, we have that
\begin{align*}
\sum_{t=1}^T (Q_t+K^{\star\top} R_t K^\star) \bullet\widehat{X}^\star-&\sum_{t=1}^T (Q_t+K^{\dagger\top} R_t K^\dagger) \bullet\widehat{X}^\dagger\\
&=T (\bar{Q}_T+K^{\star\top} \bar{R}_T K^\star)\bullet\widehat{X}^\star-T (\bar{Q}_T+K^{\dagger\top} \bar{R}_T K^\dagger)\bullet\widehat{X}^\dagger\\
&=T(P^\star-{A^\star}^\top P^\star A^\star)\bullet \widehat{X}^\star-T(P^\dagger-{A^\dagger}^\top P^\dagger A^\dagger)\bullet \widehat{X}^\dagger\\
&=T(P^\star\bullet \widehat{X}^\star- P^\star\bullet{A^\star} \widehat{X}^\star {A^\star}^\top)- T(P^\dagger\bullet \widehat{X}^\dagger- P^\dagger\bullet{A^\dagger} \widehat{X}^\dagger {A^\dagger}^\top)\\
&=T(P^\star\bullet \widehat{X}^\star-P^\star\bullet (\widehat{X}^\star-W))-T(P^\dagger\bullet \widehat{X}^\dagger-P^\dagger\bullet (\widehat{X}^\dagger-W)\\
&=T(P^\star-P^\dagger)\bullet W\leq 0,
\end{align*}
where $P^\star$ and $P^\dagger$ satisfies $P=(A-BK)^\top P(A-BK)+\bar{Q}_T +K^\top \bar{R}_T K$ for $K=K^\star$ and $K=K^\dagger$, respectively, and we have used this fact in the second equality, the cyclic property of the trace in the third equality, and \eqref{eqre1} in the forth equality. By \cite[Theorem 1]{GH:71}, $P^\star\preceq P^\dagger$ and we have the result.

We start with our first technical result, which shows that Algorithm~\ref{alg} produces stable polices. This step is similar to the classical result of~\cite{GH:71} for the case where the cost objective matrices $Q_t$ and $R_t$ are fixed. Recall that stability of policies $K_t$ is required to establish strong stability, see Proposition~\ref{str-stability}.

\begin{lemma}\label{stability}
Suppose that the pair $(A,B)$ is stabilizable and let the sequence $\{K_t\}_{t\geq 1}$ be generated by Algorithm~\ref{alg}, starting from a stable policy $K_1$. Then policy $K_t$ remains stable for all $t\geq 1$.
\end{lemma}
\begin{proof}
We proceed by an induction argument. First, since the system is stabilizable, there exists a stable policy and hence we can choose $K_1$ to be stable, i.e. such that $\rho(A-BK_1)< 1$. Assume now that $K_t$ is stable, for some $ t\geq 1 $. 
Then, using~\eqref{eq31}, $P_t$ is uniquely determined by 
\begin{equation}\label{eq3}
P_t=\sum_{i=0}^\infty (A_t^{\top})^i(\bar{Q}_t+K_t^\top \bar{R}_{t} K_t)A_t^{i}.
\end{equation}
By a straightforward computation, we have that
\begin{align*}
A_t^\top P_t A_t +K_t^\top \bar{R}_{t}K_t=&(A-BK_t)^\top P_t (A-BK_t)+K_t^\top \bar{R}_{t}K_t\\
=& A^\top P_t A-K_t^\top B^\top P_tA-A^\top P_t BK_t +K_t^\top(B^\top P_t B+\bar{R}_t)K_t\\
=& A^\top P_t A-K_t^\top (B^\top P_t B+\bar{R}_t)K_{t+1}-K_{t+1}^\top (B^\top P_t B+\bar{R}_t)K_t\\
&+K_t^\top(B^\top P_t B+\bar{R}_t)K_t\\
=& A^\top P_t A+(K_{t+1}-K_t)^\top (B^\top P_t B+\bar{R}_{t})(K_{t+1}-K_t)\\
&-K_{t+1}^\top(B^\top P_t B+\bar{R}_t)K_{t+1}\\
=& A^\top P_t A+(K_{t+1}-K_t)^\top (B^\top P_t B+\bar{R}_{t})(K_{t+1}-K_t)\\
&-K_{t+1}^\top B^\top P_t A-A^\top P_t B K_{t+1}+K_{t+1}^\top(B^\top P_t B+\bar{R}_t)K_{t+1}\\
=& A_{t+1}^\top P_t A_{t+1}+K_{t+1}^\top \bar{R}_{t}K_{t+1}+(K_{t+1}-K_t)^\top (B^\top P_t B+\bar{R}_{t})(K_{t+1}-K_t),
\end{align*}
where we have used $(B^\top P_t B+\bar{R}_t)K_{t+1}=B^\top P_t A$ in the third and fifth equalities. 
Therefore, using this and~\eqref{eq31}, we have that
\begin{equation}\label{eq10}
P_t=A_{t+1}^\top P_t A_{t+1}+ V,
\end{equation}
where 
\[
V=K_{t+1}^\top \bar{R}_{t}K_{t+1}+(K_{t+1}-K_t)^\top (B^\top P_t B+\bar{R}_{t})(K_{t+1}-K_t)+\bar{Q}_t.
\] 
As a result, 
\begin{align}\label{eqx10}
P_t=\sum_{i=0}^\infty (A_{t+1}^{\top})^i(V)A_{t+1}^{i},
\end{align}
It is easy to observe that $V$ is positive-definite. 
Now, using \eqref{eq3}, since $K_t$ is stable, the matrix $P_t$ is finite. Using \eqref{eqx10}, and the fact that the left side of \eqref{eqx10} is finite, we have that $\rho(A_{t+1})<1$, i.e., $K_{t+1}$ is stable, otherwise the sum on the right side of \eqref{eqx10} will diverge.
\qed\end{proof}

In order to get a $\log(T)$ regret bound, we need to have bounds  of order $\mathcal{O}(1/t)$ on $\|P_t-P_{t-1}\|$, $\|\widehat{X}_t-\widehat{X}_{t-1}\|$ and $\|K_t-K_{t-1}\|$. Also, recall that such bounds are essential for obtaining sequential strong stability using  Proposition~\ref{lem6}. 
The next lemma and its corollary serves this purpose. 
\begin{lemma}\label{lem7}
Suppose that $\mu I\preceq Q_t, R_t$ and $\Tr(Q_t), \Tr(R_t)\leq \sigma$. Let $\{P_t\}_{t\geq 1}$ and $\{K_t\}_{t\geq 1}$ be the sequences of matrices generated by Algorithm~\ref{alg}, and assume that the sequence $\{K_t\}_{t\geq 1}$ is $(\kappa,\gamma)$-strongly stable.
Then we have $\|P_{t+1}-P_{t}\|\leq m/t$ for some $m>0$, for $t\geq 1$.
\end{lemma}

\begin{proof}
Note that using~\eqref{eq10}, we have
\begin{align}\label{eq6}
P_{t+1}-P_t=&A_{t+1}^\top(P_{t+1}-P_t)A_{t+1}+K_{t+1}^\top(\bar{R}_{t+1}-\bar{R}_t)K_{t+1}+(\bar{Q}_{t+1}-\bar{Q}_t)\nonumber\\
&-(K_{t+1}-K_t)^\top(B^\top P_t B+\bar{R}_t)(K_{t+1}-K_t).
\end{align}
By the definition of $K_{t}$, we have the following identity:
\begin{align}\label{eq8}
K_{t+1}-K_t=(B^\top P_t B+\bar{R}_t)^{-1}\big[B^\top (P_t - P_{t-1})A_{t}+(\bar{R}_{t-1}-\bar{R}_t)K_t\big].
\end{align}
Using this along with~\eqref{eq6}, we have that
\begin{align}
P_{t+1}-P_t=&A_{t+1}^\top(P_{t+1}-P_t)A_{t+1}+K_{t+1}^\top(\bar{R}_{t+1}-\bar{R}_t)K_{t+1}+(\bar{Q}_{t+1}-\bar{Q}_t)\nonumber\\
&-\big[B^\top (P_t - P_{t-1})A_{t}+(\bar{R}_{t-1}-\bar{R}_t)K_t\big]^\top(B^\top P_t B+\bar{R}_t)^{-1} \\
& \quad \times\big[B^\top (P_t - P_{t-1})A_{t}+(\bar{R}_{t-1}-\bar{R}_t)K_t\big].
\end{align}

By the stability of $K_{t+1}$, we have that
\begin{align}\label{eq7}
P_{t+1}-P_t=&\sum_{i=0}^\infty (A^\top_{t+1})^i M_t A^i_{t+1}\nonumber\\
&\leq \|M_t\|\sum_{i=0}^\infty (A^\top_{t+1})^i A^i_{t+1},
\end{align}
where 
\begin{align*}
M_t=&K_{t+1}^\top(\bar{R}_{t+1}-\bar{R}_t)K_{t+1}+(\bar{Q}_{t+1}-\bar{Q}_t)\\
&-\big[B^\top (P_t - P_{t-1})A_{t}+(\bar{R}_{t-1}-\bar{R}_t)K_t\big]^\top(B^\top P_t B+\bar{R}_t)^{-1}\\
&\quad\times \big[B^\top (P_t - P_{t-1})A_{t}+(\bar{R}_{t-1}-\bar{R}_t)K_t\big].
\end{align*}
Given the strong stability of $K_{t+1}$, we can write $A_{t+1}=H_{t+1}L_{t+1}H_{t+1}^{-1}$. Hence, we have that
\begin{align*}
\|\sum_{i=0}^\infty (A^\top_{t+1})^i A^i_{t+1}\|&\leq \sum_{i=0}^\infty \|(A^\top_{t+1})^i A^i_{t+1}\|\\
&\leq\sum_{i=0}^\infty \|H_{t+1}\|^2\|H_{t+1}^{-1}\|^2\|L_{t+1}\|^{2i}\\
&\leq\sum_{i=0}^\infty \kappa^2 (1-\gamma)^{2i} = \frac{\kappa^2}{1-(1-\gamma)^2}\leq \frac{\kappa^2}{\gamma},
\end{align*}
where we used $\|H_{t+1}\|\|H_{t+1}^{-1}||\leq \kappa$ and $\|L_{t+1}\|\leq 1-\gamma$.
We now proceed to bound $M_t$. We can write
\begin{align}\label{Mbound}
\|M_t\|=&\|K_{t+1}^\top(\bar{R}_{t+1}-\bar{R}_t)K_{t+1}+(\bar{Q}_{t+1}-\bar{Q}_t)\|\nonumber\\
&+\|(B^\top P_t B+\bar{R}_t)^{-1}\|(\|B\|\|A_{t}\|\|P_{t}-P_{t-1}\|+\|(\bar{R}_{t-1}-\bar{R}_t)K_t\|)^2.
\end{align}
Using~\eqref{eq7} and \eqref{Mbound}, we also have
\begin{equation}
z_{t+1}\leq c_t(h_t z_t+d_t)^2+e_{t+1},
\end{equation}
where $z_t=\|P_t-P_{t-1}\|$, and 
\begin{align*}
c_t&=\frac{\kappa^2}{\gamma}\|(B^\top P_t B+\bar{R}_t)^{-1}\|\\
d_t&=\|(\bar{R}_{t-1}-\bar{R}_t)K_t\|\\
e_{t+1}&=\frac{\kappa^2}{\gamma}\|K_{t+1}^\top(\bar{R}_{t+1}-\bar{R}_t)K_{t+1}+(\bar{Q}_{t+1}-\bar{Q}_t)\|\\
h_t&=\|B\|\|A_{t}\|.
\end{align*}
Using the fact that 
\[
\|\bar{Q}_{t+1}-\bar{Q}_t\|=\frac{1}{t+1}\|(Q_t-\bar{Q}_t)\|\leq \frac{2}{t+1}\max_{t\geq 0}\|Q_t\|\leq \frac{2\sigma}{t+1},
\]
along with
\[
\|\bar{R}_{t+1}-\bar{R}_t\|=\frac{1}{t+1}\|(R_t-\bar{R}_t)\|\leq \frac{2}{t+1}\max_{t\geq 0}\|R_t\|\leq \frac{2\sigma}{t+1},
\] 
and \[\|(B^\top P_t B+\bar{R}_t)^{-1}\|\leq (\lambda_{\min} (R_t))^{-1}\leq \mu^{-1},\]
and $\|A_t\|\leq \kappa$, we conclude 
\begin{equation}\label{bbound}
c_t\leq \kappa^2/\gamma\mu, \ d_t\leq \frac{2\sigma\kappa}{t}, \ \mathrm{and} \ e_t\leq \frac{2\kappa^2\sigma(1+\kappa^2)}{\gamma t}, \ h_t \leq \|B\|\kappa,
\end{equation}
for $t \geq 1$.
We next claim that there exists a time $t^{\star}$ and a constant $m>0$ such that $z_{t}\leq m/t$ for all $t> t^\star$. We use an inductive argument to prove this statement. The base case will be proved later. Assume now that $z_{t}\leq m/t$; we show that $z_{t+1}\leq m/(t+1)$. First, note that if 
\[
m\leq \frac{2\sigma}{\|B\|}+\frac{4\kappa^2\sigma(1+\kappa^2)}{\gamma},
\]
for 
$t\geq t^\star=\frac{4\kappa^3\|B\|}{\gamma\mu}(2\sigma\kappa+\frac{2\kappa^3\|B\|\sigma(1+\kappa^2)}{\gamma})+1$, using an elementary calculation, one can observe that 
\begin{align*}
\frac{\kappa^2}{\gamma\mu}(\kappa\|B\|\frac{m}{t}+\frac{2\sigma\kappa}{t})^2+\frac{2\kappa^2\sigma(1+\kappa^2)}{\gamma(t+1)}\leq \frac{m}{t+1}.
\end{align*} 
The claim then follows by noting that 
\begin{align*}
z_{t+1}\leq c_t(h_t z_t+d_t)^2+e_t\leq \frac{\kappa^2}{\gamma\mu}(\kappa\|B\|\frac{m}{t}+\frac{2\sigma\kappa}{t})^2+\frac{2\kappa^2\sigma(1+\kappa^2)}{\gamma(t+1)},
\end{align*} 
where we have used~\eqref{bbound}.

It remains to show that the condition we placed to obtain the last inequality, i.e., that $z_{t^\star+1}\leq m/(t^\star+1)$, is satisfied. To proceed with this, first note that $ t^\star $ is exactly the reset time in Algorithm~\ref{alg}. Also, the evolution of  $\widehat{P}_\ell$ in the reset part of the algorithm is still according to~\eqref{eq6}. Since the matrices $Q_t$ and $R_t$ are fixed in the reset part, $\{\widehat{P}_{\ell}\}$ is a Cauchy sequence. Hence, by choosing $\ell$ large enough, we have that $\|\widehat{P}_{\ell}-\widehat{P}_{\ell-1}\|\leq m/t^\star$, terminating the reset stage of the algorithm; with slight abuse of notation, we let $ \widehat{P}_{\ell} $ be the outcome of the reset part of the algorithm. Note that at time $ t^\star\ $ the algorithm implements $ P_{t^\star}= \widehat{P}_{\ell}$. In the next time step $t^\star+1$, the algorithm updates $P_{t^\star+1}$ as usual, using~\eqref{eq31}. We know by the previous part of the proof that $\|P_{t^\star+1}-P_{t^\star}\|\leq m/(t^\star+1)$, which shows that $z_{t^\star+1}\leq m/(t^\star+1)$ is satisfied. To conclude the proof, note that we can show that $z_{t}\leq \hat{m}/t$, for all $t\geq 1$, simply by selecting $\hat{m}=\max\{m,tz_{t}|t\leq t^\star\}$.
\qed\end{proof}

\begin{corollary}\label{corol}
Let $\widehat{X}_t$ be the steady-state covariance matrix using policy $K_t$ generated by Algorithm~\ref{alg}. Then we have $\|\widehat{X}_t-\widehat{X}_{t-1}\|\leq M/t+M'/t^2$ for some $M>0$ and $M'>0$ and for $t \geq 1$.
\end{corollary}
\begin{proof}
By the definition of $\widehat{X}_t$, we have that
\begin{align*}
\widehat{X}_t-\widehat{X}_{t-1}=& A_t \widehat{X}_t A_t^\top - A_{t-1}\widehat{X}_{t-1} A^\top_{t-1}\\
=& A_t (\widehat{X}_t - \widehat{X}_{t-1}) A^\top_t + (A_t-A_{t-1})\widehat{X}_{t-1} (A_t-A_{t-1})^\top \\
&+A_{t-1} \widehat{X}_{t-1}(A_{t}-A_{t-1})^\top+(A_{t}-A_{t-1})\widehat{X}_{t-1}A_{t-1}^\top\\
=& A_t (\widehat{X}_t - \widehat{X}_{t-1}) A^\top_t +B(K_t-K_{t-1})\widehat{X}_{t-1}(K_t-K_{t-1})^\top B^\top\\
&+A_{t-1} \widehat{X}_{t-1}(K_{t-1}-K_{t})^\top B^\top +B(K_{t-1}-K_{t})\widehat{X}_{t-1} A_{t-1}.
\end{align*}
Note that Lemma \ref{lem7} can be used to bound $K_t-K_{t-1}$. Using \eqref{eq8}, we have that
\begin{align}
\|K_{t+1}-K_t\|\leq &\|(B^\top P_t B+\bar{R}_t)^{-1}\|\big[\|B\|\|P_t - P_{t-1}\|\|A_{t}\|+\|\bar{R}_{t-1}-\bar{R}_t\|\|K_t\|\big]\nonumber\\
\leq & \frac{\kappa}{\mu}(\|B\|\hat{m}+2\sigma)/t\label{eq20},
\end{align}
where we have used $\|(B^\top P_t B+\bar{R}_t)^{-1}\|\leq \mu^{-1}$, $\|A_t\|\leq \kappa$, $\|K_t\|\leq \kappa$, and $\hat{m}$ is given in the proof of Lemma~\ref{lem7}. 
Using this $\|\widehat{X}_t - \widehat{X}_{t+1}\| $ is bounded by $M/t+M'/t^2$, where \begin{equation}\label{M}
M=\frac{2\kappa^6\omega}{\mu\gamma^2}\|B\|\big(\|B\|\hat{m}+2\sigma),
\end{equation}
and
\begin{equation}\label{Mpr}
M'=\frac{\kappa^6\omega}{\mu\gamma^2}\|B\|^2(\|B\|\hat{m}+2\sigma)^2,
\end{equation}
where we have used \[\|\widehat{X}_{t-1}\|\leq\|W\|\sum_{i=0}^\infty \|(A_{t-1}^\top)^i (A_{t-1}^i)\|\leq \frac{\omega\kappa^2}{\gamma}\]
\qed\end{proof}
The following lemmas will be used to derive bounds on the redundancy terms~\eqref{eqx1}, \eqref{eqx2}, and \eqref{eqx3}. 
\begin{lemma}\label{lem8}
Suppose that the tuple $(A,B,\{Q_t\}_{t=1}^T,\{R_t\}_{t=1}^T,W)$ satisfies Assumption~\ref{assump}. Suppose that the matrices $P_t$ generated by Algorithm~\ref{alg} are uniformly bounded, i.e., $P_t\leq \nu I$. Let $\kappa=\sqrt{\frac{\nu}{\mu}}$ and $\gamma=1/2\kappa^2$. Then for the covariance matrices $X_t$ and $\widehat{X}_t$, we have
\begin{align*}
\sum_{t=1}^T (Q_t+K_t^\top R_t K_t)\bullet (X_t-\widehat{X}_t)\leq& t^\star \sigma(1+\kappa^2)\max_{0<t\leq t^\star}\|(X_t-\widehat{X}_t)\|\\
&+2\kappa^4\sigma \bigg(\|X_{t^\star}-\widehat{X}_{t^\star}\|\frac{e^{-2\gamma^2 t^\star}}{1-e^{-2\gamma^2}}+\frac{M'\pi^2}{6(1-e^{-2\gamma^2})}\\
&+\frac{M}{1-e^{-2\gamma^2}}\log\Big(\frac{T}{t^\star}\Big)\bigg).
\end{align*}
\end{lemma}
\begin{proof}
For $t\geq t^\star$, we have that $\|P_{t+1}-P_t\|\leq m/t\leq \mu/\kappa^2$. Then, using Proposition~\ref{lem6}, the matrices $K_t$ are sequentially $(\kappa,\gamma)$-strongly stable for $t\geq t^\star$ ($\kappa=\sqrt{\frac{\nu}{\mu}}$ and $\gamma=1/(2\kappa^2)$). Using this by Lemma~\ref{lem3}, we conclude that for $t\geq t^\star$
\begin{align}\label{eq11}
\|X_{t+1}-\widehat{X}_{t+1}\|\leq \kappa^2e^{-2\gamma^2 (t+1-t^\star)}\|X_{t^\star}-\widehat{X}_{t^\star}\|+\kappa^2\sum_{s=0}^{t-t^\star}e^{-2\gamma^2 s}\eta_{t-s};
\end{align}
hence we can separate \eqref{eqx1} 
into two parts as follows:
\begin{align*}
\sum_{t=1}^T (Q_t+K_t^\top R_t K_t)\bullet (X_t-\widehat{X}_t)=&\sum_{t=1}^{t^\star}(Q_t+K_t^\top R_t K_t)\bullet (X_t-\widehat{X}_t)\\
&+\sum_{t=t^\star}^{T}(Q_t+K_t^\top R_t K_t)\bullet (X_t-\widehat{X}_t).
\end{align*}
By stability of policies $K_t$, the matrices $X_t$ and $\widehat{X}_t$ are bounded and we have that
\begin{align}
\sum_{t=1}^{t^\star}(Q_t+K_t^\top R_t K_t)\bullet (X_t-\widehat{X}_t)= & \sum_{t=1}^{t^\star}\Tr\big[(Q_t+K_t^\top R_t K_t)(X_t-\widehat{X}_t)\big]\nonumber\\
\leq & \sum_{t=1}^{t^\star}\Tr(Q_t+K_t^\top R_t K_t)\|(X_t-\widehat{X}_t)\|\nonumber\\
\leq & t^\star \sigma(1+\kappa^2)\max_{0<t\leq t^\star}\|(X_t-\widehat{X}_t)\|\label{equ3},
\end{align}
where we have used $$\Tr(Q_t+K_t^\top R_t K_t)\leq \sigma(1+\kappa^2)$$

Using \eqref{eq11}, we have that
\begin{align*}
\sum_{t=t^\star}^{T}(Q_t+K_t^\top R_t K_t)\bullet (X_t-\widehat{X}_t) \leq & \sum_{t=t^\star}^{T}\Tr(Q_t+K_t^\top R_t K_t)\|(X_t-\widehat{X}_t)\\
\leq & \sum_{t=t^\star}^{T}\sigma(1+\kappa^2)\|X_{t}-\widehat{X}_{t}\|\\
\leq & (\sigma(1+\kappa^2))\kappa^2\sum_{t=t^\star}^{T}\Big(e^{-2\gamma^2 t}\|X_{t^\star}-\widehat{X}_{t^\star}\|+ \sum_{s=0}^{t-t^\star}e^{-2\gamma^2 s}\eta_{t-s}\Big)\\
\leq &2\kappa^4\sigma (\|X_{t^\star}-\widehat{X}_{t^\star}\|\frac{e^{-2\gamma^2 t^\star}}{1-e^{-2\gamma^2}}+\sum_{t=t^\star}^{T}\sum_{s=0}^{t-t^\star}e^{-2\gamma^2 s}\eta_{t-s}).
\end{align*}
Note that by using Corollary \ref{corol}, we have $\eta_t=M/t+M'/t$, where $M$ and $M'$ are given by \eqref{M} and \eqref{Mpr}. Consequently,
\begin{align}
\sum_{t=t^\star}^{T}(Q_t+K_t^\top R_t K_t)\bullet (X_t-\widehat{X}_t) \leq & 2\kappa^4\sigma \|X_{t^\star}-\widehat{X}_{t^\star}\|\frac{e^{-2\gamma^2 t^\star}}{1-e^{-2\gamma^2}}\nonumber\\
&+2\kappa^4\sigma\sum_{t=t^\star}^{T}\sum_{s=0}^{t-t^\star}e^{-2\gamma^2 s}\Big(\frac{M}{t-s}+\frac{M'}{(t-s)^2}\Big)\label{equ1}.
\end{align}
Next, by changing the order of summation we obtain
\begin{align*}
\sum_{t=t^\star}^{T}\sum_{s=0}^{t-t^\star}e^{-2\gamma^2 s}\Big(\frac{M}{t-s}+\frac{M'}{(t-s)^2}\Big)=&\sum_{s=0}^{T-t^\star}e^{-2\gamma^2 s}\sum_{t=s+t^\star}^{T}\Big(\frac{M}{t-s}+\frac{M'}{(t-s)^2}\Big)\\
\leq & \sum_{s=0}^{T-t^\star}e^{-2\gamma^2 s}\Big(M\log\Big(\frac{T-s}{t^\star}\Big)+\frac{M'\pi^2}{6}\Big)\\
\leq & \frac{M'\pi^2}{6(1-e^{-2\gamma^2})}+\sum_{s=0}^{T-t^\star}Me^{-2\gamma^2 s}\log\Big(\frac{T}{t^\star}\Big)\\
\leq & \frac{M'\pi^2}{6(1-e^{-2\gamma^2})}+\frac{M}{1-e^{-2\gamma^2}}\log\Big(\frac{T}{t^\star}\Big), 
\end{align*}
where we have used a logarithmic upper bound for $\sum_{t=t^\star}^{T-s}1/t$ and the identity $\sum_{t=1}^\infty 1/t^2=\pi^2/6$ in the second inequality. The third and fourth inequalities follow by manipulating geometric series.
Therefore, by substituting this inequality in Equation~\eqref{equ1} we obtain 
\begin{align}
\sum_{t=t^\star}^{T}(Q_t+K_t^\top R_t K_t)\bullet (X_t-\widehat{X}_t) \leq &2\kappa^4\sigma \Big(\|X_{t^\star}-\widehat{X}_{t^\star}\|\frac{e^{-2\gamma^2 t^\star}}{1-e^{-2\gamma^2}}+\frac{M'\pi^2}{6(1-e^{-2\gamma^2})}\nonumber\\
&+\frac{M}{1-e^{-2\gamma^2}}\log\Big(\frac{T}{t^\star}\Big)\Big)\label{equ2} 
\end{align}

The result follow by adding \eqref{equ3} and \eqref{equ2}.

\qed\end{proof}
\begin{lemma}\label{lem9}
Suppose that the tuple $(A,B,\{Q_t\}_{t=1}^T,\{R_t\}_{t=1}^T,W)$ satisfies Assumption~\ref{assump}. Suppose that the matrices $P_t$ generated by Algorithm~\ref{alg} are uniformly bounded, i.e., $P_t\leq \nu I$. Let $\kappa=\sqrt{\frac{\nu}{\mu}}$ and $\gamma=1/2\kappa^2$. Then the covariance matrices $\widehat{X}_t$ and $\widehat{X}^\star$ satisfy
\begin{align*}
\sum_{t=1}^T (Q_t+K_t^\top R_t K_t)&\bullet\widehat{X}_t-\sum_{t=1}^T (Q_t+K^{\star\top} R_t K^\star) \bullet\widehat{X}^\star\leq \omega l\hat{m} \\
&+\frac{\kappa^4\omega}{\gamma\mu^3} (\|B\|\hat{m}+2\sigma)^2 (1+\log(T)).
\end{align*}
\end{lemma}
\begin{proof}
Using the fact that $Q_t=t\bar{Q}_t-(t-1)\bar{Q}_{t-1}$ and $R_t=t\bar{R}_t-(t-1)\bar{R}_{t-1}$, we have
\begin{align}
(Q_t+K_t^\top R_t K_t)=&(t\bar{Q}_t-(t-1)\bar{Q}_{t-1})+K_t^\top (t\bar{R}_t-(t-1)\bar{R}_{t-1}) K_t\nonumber\\
=&t(\bar{Q}_t+K_t^\top\bar{R}_t K_t)-(t-1)(\bar{Q}_{t-1}+K_t^\top\bar{R}_{t-1} K_t)\nonumber\\
=&t(P_t-A_t^\top P_t A_t)-(t-1)(P_{t-1}-A_{t}^\top P_{t-1}A_t)\nonumber\\
&+(t-1)(K_t-K_{t-1})^\top (B^\top P_{t-1}B+\bar{R}_{t-1})^{-1}(K_t-K_{t-1})\label{equ4},
\end{align}
where we have used~\eqref{eq31} and~\eqref{eq10} in the third equality. Note that
\begin{align}
A_t^\top P_t A_t\bullet \widehat{X}_t=&\Tr(A_t^\top P_t A_t\widehat{X}_t)\nonumber\\
=&\Tr(P_t A_t \widehat{X}_t A_t^\top)\nonumber\\
=&P_t \bullet A_t \widehat{X}_t A_t^\top \nonumber\\
=&P_t\bullet (\widehat{X}_t-W)\nonumber\\
=&P_t \bullet \widehat{X}_t-P_t \bullet W\label{equ5}.
\end{align}
Therefore, by multiplying~\eqref{equ4} and $\widehat{X}_t$ we obtain
\begin{align*}
(Q_t+K_t^\top R_t K_t)\bullet \widehat{X}_t=&t P_t\bullet W-(t-1)P_{t-1}\bullet W \\
&+(t-1)(K_t-K_{t-1})^\top (B^\top P_{t-1}B+\bar{R}_{t-1})^{-1}(K_t-K_{t-1})\bullet \widehat{X}_t,
\end{align*}
where we have used \eqref{equ5} to cancel out some terms.
Summing over $t$ and using the telescopic series for $t P_t\bullet W-(t-1)P_{t-1}\bullet W$, we obtain
\begin{align}
\sum_{t=1}^T (Q_t+&K_t^\top R_t K_t)\bullet \widehat{X}_t\leq TP_T\bullet W \nonumber\\
&+\sum_{t=1}^T(t-1)(K_t-K_{t-1})^\top (B^\top P_{t-1}B+\bar{R}_{t-1})^{-1}(K_t-K_{t-1})\bullet \widehat{X}_t\label{equ6}.
\end{align}
On the other hand,
\begin{align}
\sum_{t=1}^T (Q_t+K^{\star\top} R_t K^\star)\bullet\widehat{X}^\star=&T (\bar{Q}_T+K^{\star\top} \bar{R}_T K^\star)\bullet \widehat{X}^\star \nonumber\\
=&T (P^\star-{A^\star}^\top P^\star A^\star)\bullet \widehat{X}^\star \nonumber\\
=&T (P^\star\bullet\widehat{X}^\star-P^\star \bullet A^\star \widehat{X}^\star {A^\star}^\top)\nonumber\\
=&T(P^\star\bullet\widehat{X}^\star-P^\star \bullet\widehat{X}^\star +P^\star \bullet W)\nonumber\\
=&TP^\star \bullet W\label{equ7}.
\end{align}
Therefore, by subtracting \eqref{equ7} from \eqref{equ6} we have
\begin{align*}
\sum_{t=1}^T& (Q_t+K_t^\top R_t K_t)\bullet \widehat{X}_t-\sum_{t=1}^T (Q_t+K^{\star\top} R_t K^\star)\bullet\widehat{X}^\star=T(P_T-P^\star)\bullet W+\\
&+\sum_{t=1}^T(t-1)(K_t-K_{t-1})^\top (B^\top P_{t-1}B+\bar{R}_{t-1})^{-1}(K_t-K_{t-1})\bullet \widehat{X}_t.
\end{align*}
Note that $P^\star$ is the solution of DARE when the cost matrices $Q_t=\bar{Q}_T$ and $R_t=\bar{R}_T$ are chosen to be fixed; it is the limit of the sequence $P_t$ when $Q_t$ and $R_t$ are chosen to be $\bar{Q}_T$ and $\bar{R}_T$, respectively. The rate of convergence is quadratic~\cite{GH:71}, 
i.e. there exists $C>0$ such that for all $t\geq 2$,
\begin{equation}\label{quadconv}
\|P_t-P^\star\|\leq C \|P_{t-1}-P^\star\|^2
\end{equation} 
and by a similar analysis, we also have 
\begin{equation}\|P_{t+1}-P_t\|\leq C \|P_{t}-P_{t-1}\|^2.
\end{equation} 
Here we use a similar technique to bound $\|P_T-P^\star\|$. We can update the sequence $P_t$ after time $T$ by starting at $P_T$ using~\eqref{eq31}, with $\bar{Q}_t=\bar{Q}_T$ and $\bar{R}_t=\bar{R}_T$ fixed for $t\geq T$. 
We hence have that
\begin{align*}
\|P_T-P^\star\|=&\lim_{t\to\infty}\|P_T-P_{t}\|\\
\leq&\lim_{t\to\infty}\sum_{i=0}^{t-1}\|P_{T+i}-P_{T+i+1}\|\\
\leq&\lim_{t\to\infty}\sum_{i=0}^{t-1}C^{2^i-1}\|P_{T}-P_{T+1}\|^{2^i}\\
=&\|P_{T}-P_{T+1}\|\lim_{t\to\infty}\sum_{i=0}^{t-1}C^{2^i-1}\|P_{T}-P_{T+1}\|^{2^i-1},
\end{align*}
where we have used~\eqref{quadconv}. 
For $T\geq t^\star$, $C\|P_{T}-P_{T+1}\|<1$ and thus the sum $\sum_{i=0}^{\infty}C^{2^i-1}\|P_{T}-P_{T+1}\|^{2^i-1}$  is bounded by some finite value $l>0$.
Hence we have 
\begin{align}
T(P_T-P^\star)\bullet W\leq T\omega\|P_T-P^\star\|\leq T\omega l \|P_{T}-P_{T+1}\|\leq \frac{T\omega l\hat{m}}{T}=\omega l\hat{m}\label{eq13},
\end{align}
where we have used $\omega=\Tr(W)$ and $\|P_{T}-P_{T+1}\|\leq \hat{m}/T$ by Lemma~\ref{lem7}.
We now proceed by noting that
\begin{align}
\sum_{t=2}^T(t-1)&(K_t-K_{t-1})^\top (B^\top P_{t-1}B+\bar{R}_{t-1})^{-1}(K_t-K_{t-1})\bullet \widehat{X}_t\nonumber\\ \leq&\sum_{t=2}^T(t-1)\Tr(\widehat{X}_t)\frac{1}{(t-1)^2\mu^3}(\|B\|\kappa\hat{m}+2\sigma\kappa)^2\nonumber\\
\leq&\frac{\kappa^2}{\gamma}\omega (\|B\|\kappa\hat{m}+2\sigma\kappa)^2\sum_{t=2}^T\frac{1}{(t-1)\mu^3}\nonumber\\
\leq & \frac{\kappa^4\omega}{\gamma\mu^3} (\|B\|\hat{m}+2\sigma)^2 (1+\log(T))\label{eq12},
\end{align}
where we have used the bound in~\eqref{eq20} on $\|K_t-K_{t-1}\|$, and the bound for $\Tr(\widehat{X}_t)$. 
Adding \eqref{eq12} and \eqref{eq13} completes the proof.
\qed\end{proof}

\begin{lemma}\label{lem10}
Suppose that the tuple $(A,B,\{Q_t\}_{t=1}^T,\{R_t\}_{t=1}^T,W)$ satisfies Assumption~\ref{assump}. Suppose that the matrices $P_t$ generated by Algorithm~\ref{alg} are uniformly bounded, i.e., $P_t\leq \nu I$. Let $\kappa=\sqrt{\frac{\nu}{\mu}}$ and $\gamma=1/2\kappa^2$. Then we have 
\begin{equation}\label{eq41}
\sum_{t=1}^T (Q_t+K^{\star\top} R_t K^\star)\bullet (\widehat{X}^\star-{X}_t^\star)\leq  \frac{\sigma(1+\kappa^2)\kappa^2}{1-e^{-2\gamma}}\|\widehat{X}^\star-{X}_1^\star\|.
\end{equation}
\end{lemma}
\begin{proof}
$\|K^\star\|\leq \kappa$ implies that $\Tr(Q_t+K^{\star\top} R_t K^\star)\leq \sigma(1+\kappa^2)$. Moreover, by Lemma \ref{stabconv}, we have that
\begin{align*}
\sum_{t=1}^T (Q_t+K^{\star\top} R_t K^\star)\bullet (\widehat{X}^\star-{X}_t^\star)\leq & \sigma(1+\kappa^2)\sum_{t=1}^T \|\widehat{X}^\star-{X}_t^\star\| \\
\leq &  \sigma(1+\kappa^2)\sum_{t=1}^T \kappa^2 e^{-2\gamma (t-1)}\|\widehat{X}^\star-{X}_1^\star\|\\
\leq &  \sigma(1+\kappa^2)\kappa^2 \frac{1}{1-e^{-2\gamma}}\|\widehat{X}^\star-{X}_1^\star\|,
\end{align*}
as claimed. 
\qed\end{proof}

To conclude, by summing the right hand side of~\eqref{equ3}, \eqref{equ2}, \eqref{eq12}, \eqref{eq13}, and \eqref{eq41}, we obtain the regret bound as follows,
\begin{align}
R(T)\leq&
\Big(2\kappa^4\sigma\frac{M}{1-e^{-2\gamma^2}}
+\frac{\kappa^4\omega}{\gamma\mu^3} (\|B\|\hat{m}+2\sigma)^2\Big)
\log(T)\nonumber\\
&-2\kappa^4\sigma\frac{M}{1-e^{-2\gamma^2}}\log(t^\star)+
 t^\star \sigma(1+\kappa^2)\max_{0<t\leq t^\star}\|(X_t-\widehat{X}_t)\|\nonumber\\
&+2\kappa^4\sigma \Big(\|X_{t^\star}-\widehat{X}_{t^\star}\|\frac{e^{-2\gamma^2 t^\star}}{1-e^{-2\gamma^2}}+\frac{M'\pi^2}{6(1-e^{-2\gamma^2})}\Big)
+\omega l\hat{m}\nonumber \\
&+\frac{\kappa^4\omega}{\gamma\mu^3} (\|B\|\hat{m}+2\sigma)^2
+ \frac{\sigma(1+\kappa^2)\kappa^2}{1-e^{-2\gamma}}\|\widehat{X}^\star-{X}_1^\star\|\label{eq:thm},
\end{align}
which finishes the proof of Theorem~\ref{theorem:main}

Note that the assumption of $(\kappa,\gamma)$-strongly stability in Theorem~\ref{theorem:main} will be satisfied as long as the solutions to the online Riccati equation are uniformly bounded. In particular, we do not need this assumption for the scalar case, see Proposition~\ref{prop1}.

\section{Simulation Results}
We provide simulation results for the proposed algorithm to illustrate its performance. The control system dynamics are given by $x_{t+1}=Ax_t+Bu_t+w_t$, where the pair $(A,B)$ is stabilizable, and $A\in\real^{10\times 10}$ and $B\in\real^{10\times 7}$, and $w_t$ is a Gaussian noise. The matrices $A$ and $B$ are chosen randomly with entry-wise i.i.d uniform distribution on $[-3,3]$ and $[-2,2]$ respectively. We have considered three scenarios for the cost functions. For the first experiment, the matrices $Q_t$ and $R_t$ are generated randomly with the Wishart distribution with unit variance and $20$ degrees of freedom. For the second and third experiment, we followed the experiment setting of~\cite{AC-AH-TK-NL-YM-KT:18}, where $Q_t=Q$ is fixed as the identity matrix, while $R_t$ is diagonal where some diagonal entries are $1$, while others are $r_t$. For the second experiment, we assume that $r_t$ is randomly changing over time with i.i.d uniform distribution on $[0.1, 1]$, and for the third experiment, we assume that $r_t$ is  changing over time according to a random walk restricted to $[0.1, 1]$ taking steps of size $0.1, -0.1, 0$, with probability $0.1, 0.1, 0.8$, respectively. We ran the algorithm with the stable matrix $K_0$ and $X_0=0$ to generate a sequence of matrices $K_t$, and we computed the regret and the average regret over time. We compared our results with the ones stated in~\cite{AC-AH-TK-NL-YM-KT:18}. 

Figure~\ref{exp1-regret} shows the regret over time for the online Riccati algorithm and the follow the lazy leader (FLL) algorithm given in~\cite{AC-AH-TK-NL-YM-KT:18} for the first experiment. The results show that both algorithms behave similarly. Although \cite{AC-AH-TK-NL-YM-KT:18} found a regret bound of $\mathcal{O}(\sqrt{T})$ while we have achieved a regret bound of $\mathcal{O}(\log{T})$, this simulation result is expected, since FLL uses the average cost matrices $Q_t$ and $R_t$ over time and finds the optimal $K_t$ and uses it for the next time step, and the online Riccati algorithm uses a Riccati update of the average cost matrices $Q_t$ and $R_t$ over time.

Figure~\ref{exp1-ave-regret},~\ref{exp2-ave-regret} and~\ref{exp3-ave-regret} show the average regret of the online Riccati algorithm, FLL algorithm, and the recent cost policy, where the optimal policy of recent cost matrices is used for the next time step. The graphs show that the online Riccati algorithm works well for different scenarios, and as expected the recent cost policy is not a good strategy and only works for the random walk scenario where the change in the cost function is slow, as also indicated in~\cite{AC-AH-TK-NL-YM-KT:18}, and we have plotted these for comparison.


\begin{figure*}[htb!]
 \centering
 {
    \includegraphics[width=.8\linewidth]{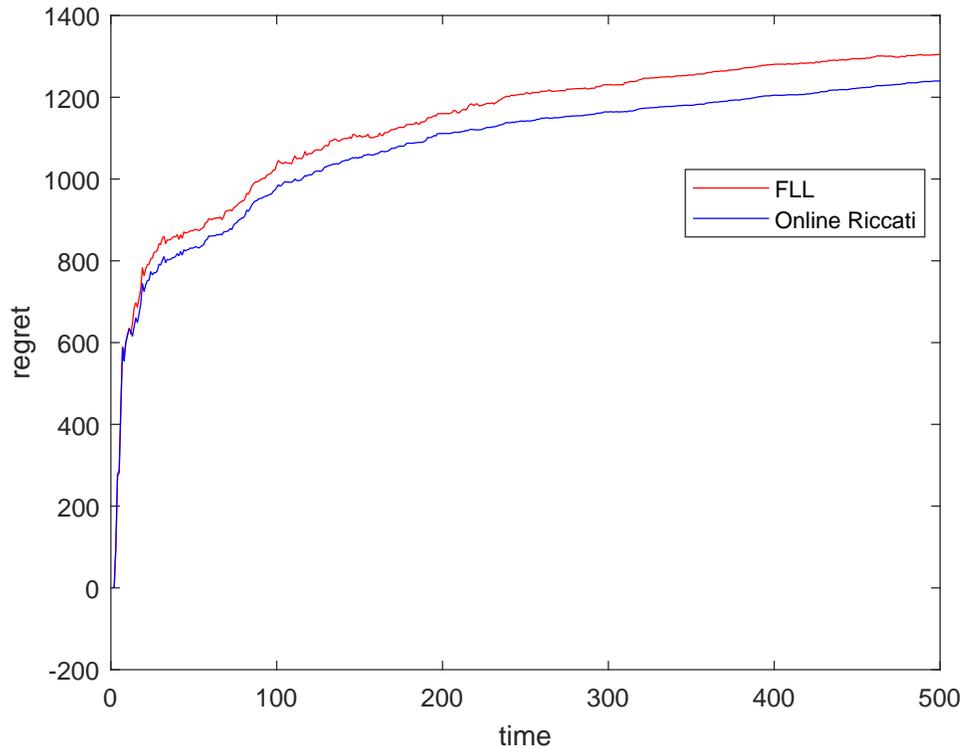}
\caption{The regret over time using the policies generated by Algorithm~\ref{alg} and FLL algorithm}\label{exp1-regret}}
\end{figure*}

\begin{figure*}[htb!]
 \centering
 {
    \includegraphics[width=.8\linewidth]{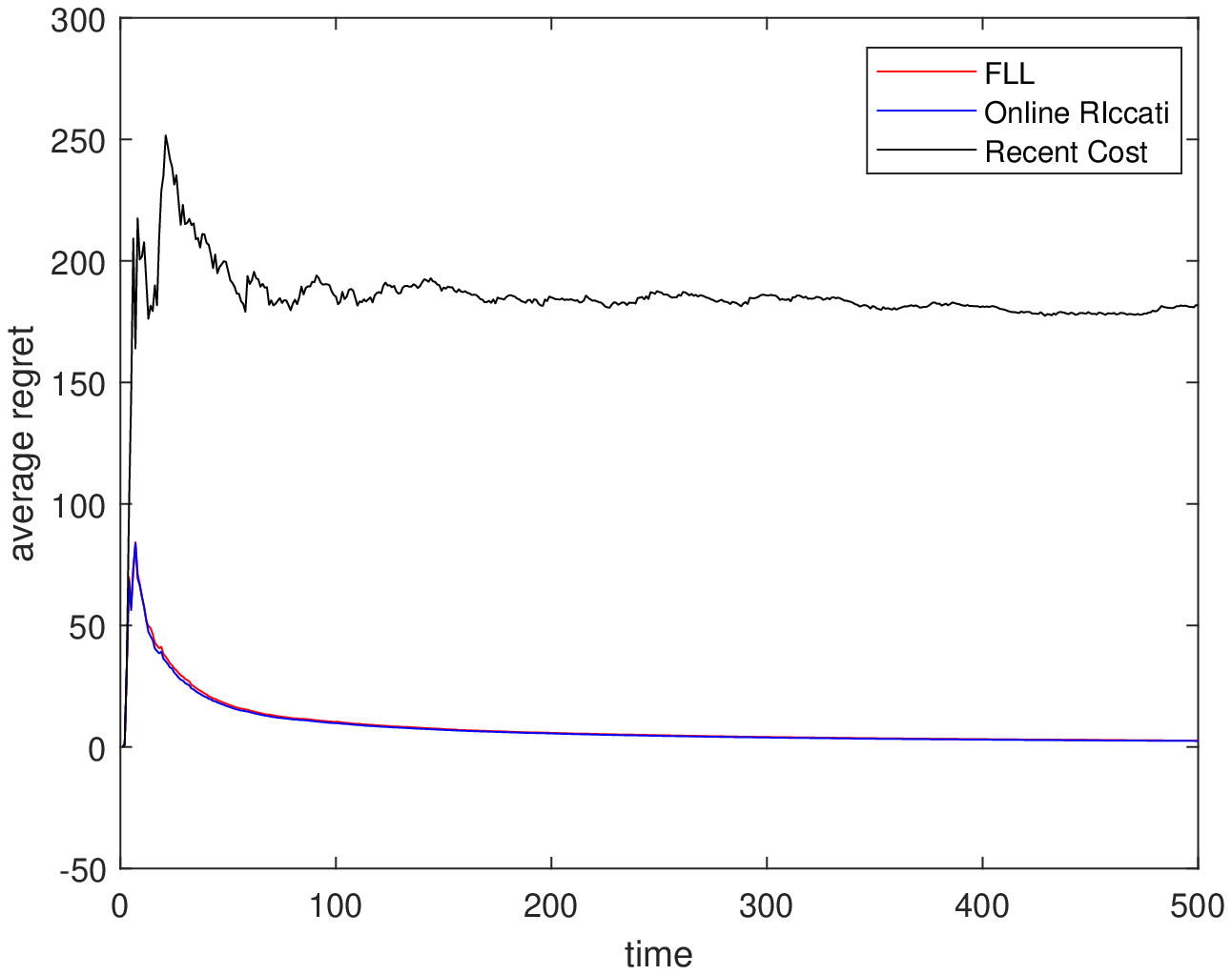}
\caption{The average regret over time using the policies generated by Algorithm~\ref{alg}, FLL algorithm and recent cost policy for the first experiment}
\label{exp1-ave-regret}}
\end{figure*}

\begin{figure*}[htb!]
 \centering
 {
    \includegraphics[width=.8\linewidth]{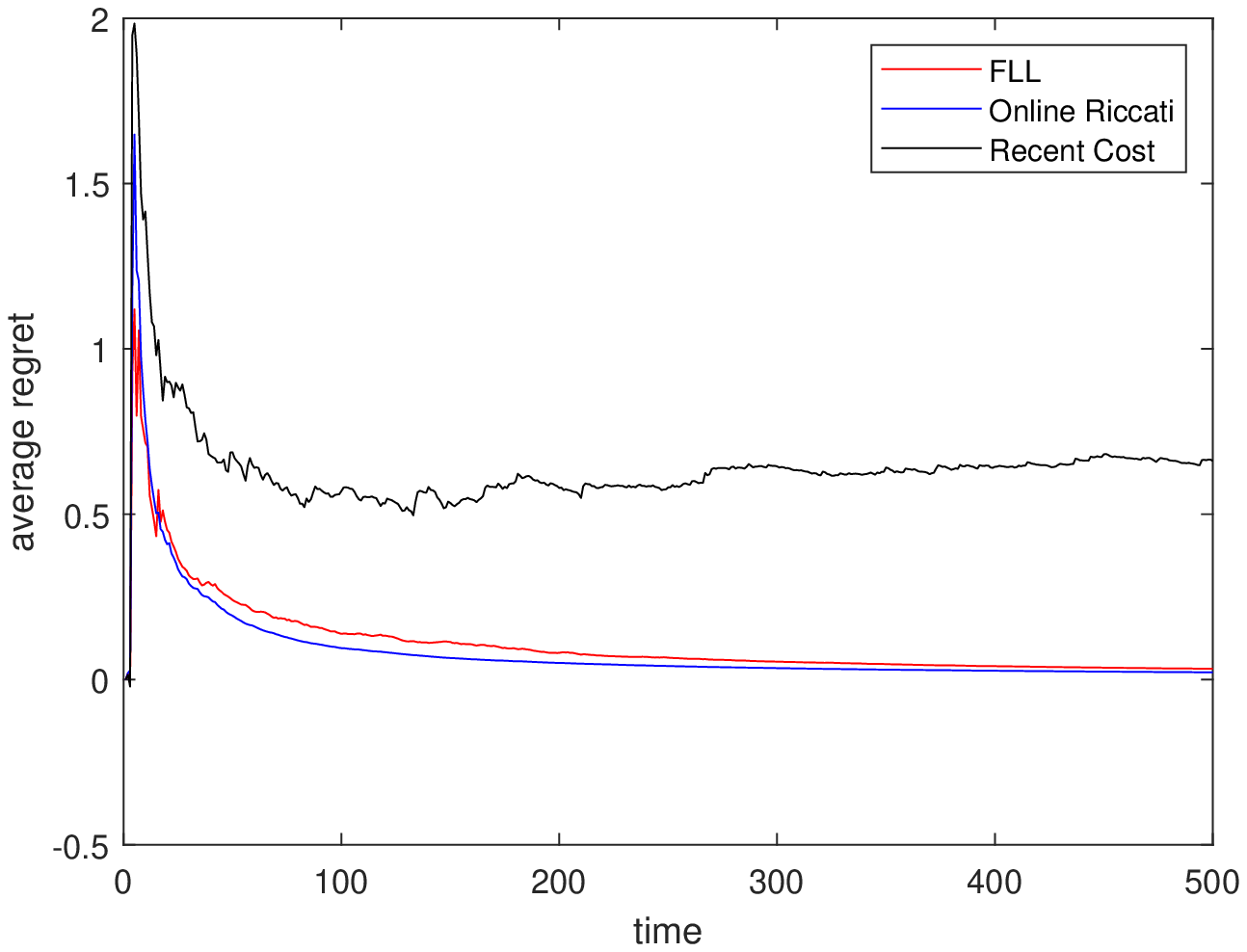}
\caption{The average regret over time using the policies generated by Algorithm~\ref{alg}, FLL algorithm and recent cost policy for the second experiment}
\label{exp2-ave-regret}}
\end{figure*}

\begin{figure*}[htb!]
 \centering
 {
    \includegraphics[width=.8\linewidth]{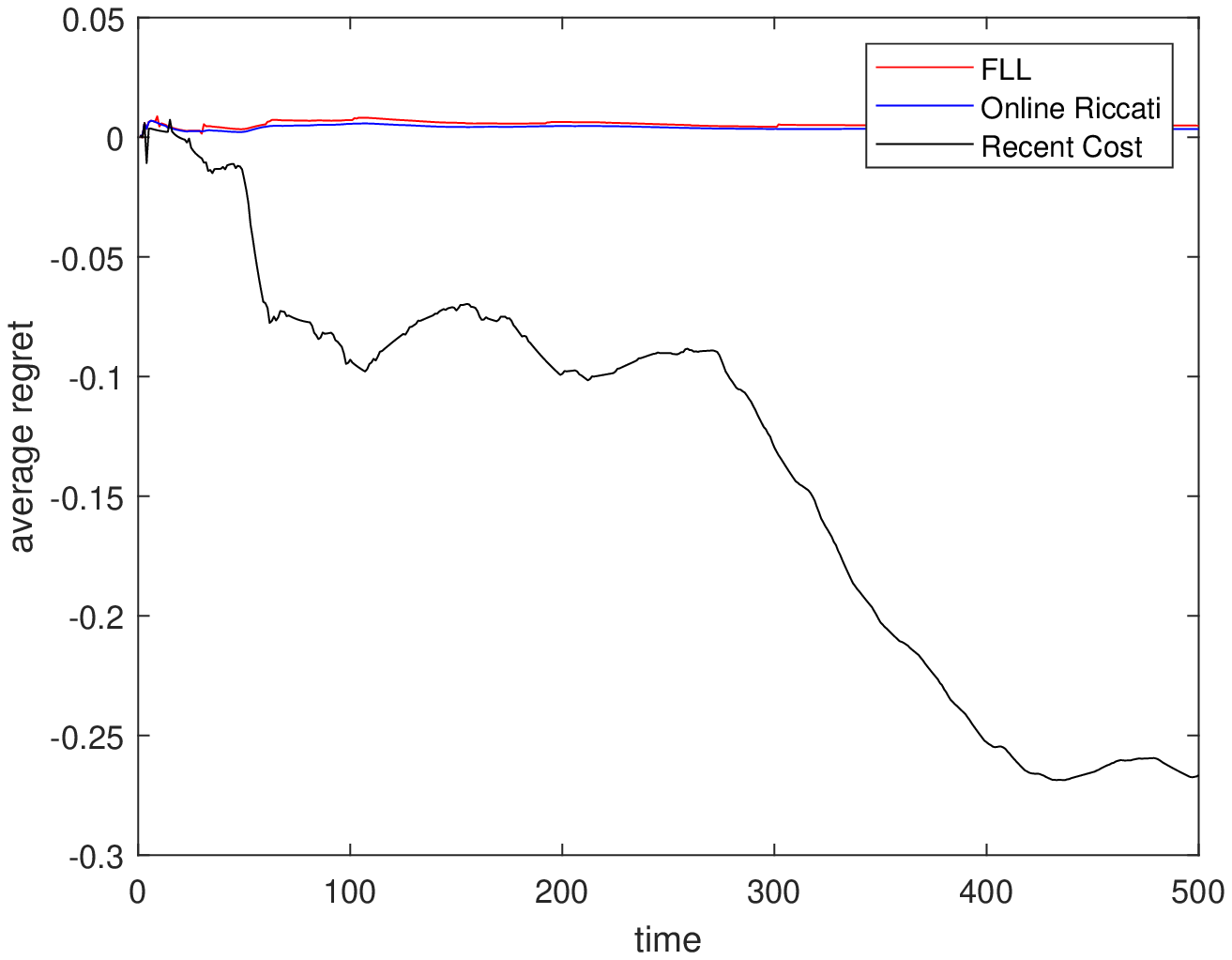}
\caption{The average regret over time using the policies generated by Algorithm~\ref{alg}, FLL algorithm and recent cost policy for the third experiment}
\label{exp3-ave-regret}}
\end{figure*}

\bibliographystyle{ieeetr}
\bibliography{alias,JC,BG,Main,MA-add,Main-add,Group-bib}

\appendix
\section{Appendix}

\newcounter{mycounter}
\renewcommand{\themycounter}{A.\arabic{mycounter}}
\newtheorem{propositionappendix}[mycounter]{Proposition}
\newtheorem{remarkappendix}[mycounter]{Remark}
\newtheorem{exampleappendix}[mycounter]{Example}

\begin{propositionappendix}\label{prop1}
Let $n=m=1$ and let $\{P_t\}_{t=1}^T$ be a sequence of positive numbers generated by Equation~\eqref{eq31} and~\eqref{eq32} recursively, and assume that policy $K_t$ is stable for all $t\geq 1$.  Then there exists $\nu>0$ such that $P_t\leq \nu $ for all $t\geq 1$.
\end{propositionappendix}

\begin{proof}
Note that
\begin{equation*}
P_t=(A-BK_t)^2 P_t+\bar{Q}_t+K_t^2 \bar{R}_{t},
\end{equation*}
Since $K_t$ is stable using the stability of $K_1$, c.f. Lemma~\ref{stability}, we have that 
\begin{align*}
P_t&=\frac{\bar{Q}_t+K_t^2 \bar{R}_{t}}{1-(A-BK_t)^2}.
\end{align*}
Now if you consider $P_t$ as a function of $K_t$, by taking derivative of $P_t$ with respect to $K_t$ and setting it to zero, we have that 
\begin{align*}
K_t=\frac{-\bar{R}_t+A^2 \bar{R}_t-B^2 \bar{Q}_t+\sqrt{(\bar{R}_t-A^2 \bar{R}_t+B^2 \bar{Q}_t)^2+4A^2B^2 \bar{R}_t \bar{Q}_t}}{2AB \bar{R}_t}
\end{align*}
minimizes the $P_t$ and the minimum admissible $P_t$ which we denote by $\tilde{P}_t$ is given by
\begin{align*}
\tilde{P}_t=\frac{A^2 \bar{R}_t-\bar{R}_t+\bar{Q}_t B^2+\sqrt{(\bar{R}_t-A^2 \bar{R}_t-\bar{Q}_t B^2)^2+4B^2 \bar{Q}_t \bar{R}_t}}{2B^2}.
\end{align*}

Now if we write $P_{t+1}$ as a function of $P_t$ we have that
\begin{align*}
P_{t+1}&=\frac{\bar{Q}_{t+1}+K_{t+1}^2 \bar{R}_{t+1}}{1-(A-BK_{t+1})^2}\\
&=\frac{\bar{Q}_{t+1}+((B^2 P_{t}+\bar{R}_{t})^{-1}BP_{t}A)^2 \bar{R}_{t+1}}{1-(A\bar{R}_{t}(B^2 P_{t}+\bar{R}_{t})^{-1})^2}\\
&=\frac{\bar{Q}_{t+1}(B^2 P_{t}+\bar{R}_{t})^2+B^2P^2_{t}A^2 \bar{R}_{t+1}}{(B^2 P_{t}+\bar{R}_{t})^2-A^2\bar{R}^2_{t}}.
\end{align*}

By taking derivative of $P_{t+1}$ with respect to $P_t$, we conclude that for the admissible $P_t$, i.e., $P_t\geq \tilde{P}_t$, the function $P_{t+1}$ is decreasing for $P_t\leq \breve{P}_t$ and increasing for $P_t\geq \breve{P}_t$~\cite{MA-BG-TL:20}, where $\breve{P}_t$ is given by
\begin{align*}
\breve{P}_t=\frac{(A^2-1) \bar{R}_{t+1}\bar{R}_t+ B^2 \bar{Q}_{t+1} \bar{R}_t+\sqrt{((A^2-1) \bar{R}_{t+1}\bar{R}_t+ B^2 \bar{Q}_{t+1} \bar{R}_t)^2+4 B^2 \bar{Q}_{t+1}\bar{R}_{t+1}\bar{R}^2_t}}{2B^2 \bar{R}_{t+1}}.
\end{align*} 
Since $P_{t+1}$ is decreasing for $P_t\leq \breve{P}_t$ and increasing for $P_t\geq \breve{P}_t$, its maximum is achieved on the boundary. So we will check the value of $P_{t+1}$ for the point $P_t$ at infinity and at its admissible minimum $\tilde{P}_t$. Now letting $P_t$ goes to infinity, we have 
\begin{align*}
P_{t+1}=\lim_{P_t\rightarrow \infty}\frac{\bar{Q}_{t+1}(B^2 P_{t}+\bar{R}_{t})^2+B^2P^2_{t}A^2 \bar{R}_{t+1}}{(B^2 P_{t}+\bar{R}_{t})^2-A^2\bar{R}^2_{t}}=\frac{A^2}{B^2} \bar{R}_{t+1}+\bar{Q}_{t+1},
\end{align*}
and for $P_t=\tilde{P}_t$, we have
\begin{align*}
P_{t+1}=\frac{\bar{Q}_{t+1}(B^2 \tilde{P}_{t}+\bar{R}_{t})^2+B^2 \tilde{P}^2_{t}A^2 \bar{R}_{t+1}}{(B^2 \tilde{P}_{t}+\bar{R}_{t})^2-A^2\bar{R}^2_{t}}
\end{align*}
One can observe that $P_{t+1}$ as a function of $R_t$ has a similar behaviour. So for $P_{t+1}$ to achieve its maximum, $(\tilde{P}_t, R_t)$ should be minimum and $(Q_{t+1},R_{t+1})$ should be maximum. So if we let $Q_{\max}=\max\{\bar{Q}_1,\bar{Q}_2,\cdots, \bar{Q}_T\}$, $Q_{\min}=\min\{\bar{Q}_1,\bar{Q}_2,\cdots, \bar{Q}_T\}$, $R_{\max}=\max\{\bar{R}_1,\bar{R}_2,\cdots,\bar{R}_T\}$, $R_{\min}=\min\{\bar{R}_1,\bar{R}_2,\cdots,\bar{R}_T\}$, and
\[
\tilde{P}_{\min}=\frac{A^2 R_{\min}-R_{\min}+Q_{\min} B^2+\sqrt{(R_{\min}-A^2 R_{\min}-Q_{\min} B^2)^2+4B^2 Q_{\min} R_{\min}}}{2B^2},
\]
we obtain that for all $t>0$
\begin{align*}
P_{t}\leq \max\Big\{\frac{A^2}{B^2} R_{\max}+Q_{\max},\frac{Q_{\max}(B^2 \tilde{P}_{\min}+R_{\min})^2+B^2 \tilde{P}^2_{\min}A^2 R_{\max}}{(B^2 \tilde{P}_{\min}+R_{\min})^2-A^2 R^2_{\min}}\Big\}
\end{align*}
\qed
\end{proof}

\begin{figure*}[htb!]
  \centering
    {
\includegraphics[scale=0.6]{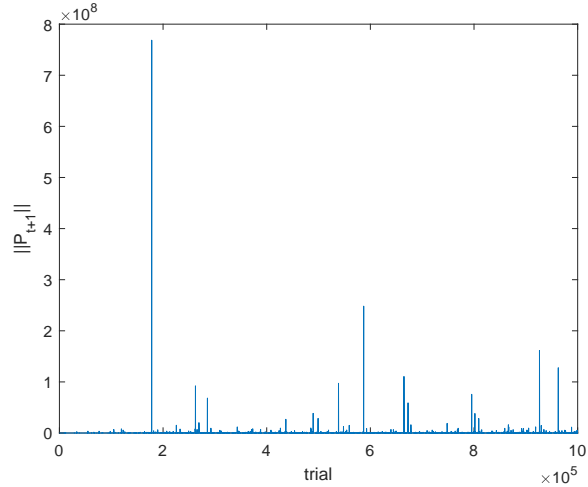}
\caption{This graph shows the norm of $P_{t+1}$ for different values of $P_t \succeq P^*$. $P_t$ can be near the boundary that makes $K_{t+1}$ unstable, and hence $P_{t+1}$ gets very large.}
\label{Example1}
}
\end{figure*}

We illustrate in the next remark as to why the argument that we have used above cannot be readily extended to non-scalar cases. 
\begin{remarkappendix}\label{remark:A}
{\em
The procedure that we have used above to prove boundedness of $ P_t $ relied on studying the evolutions of $ P_{t+1} $ as a function of $ P_t $. When these quantities are not scalars, one naturally aims to consider the norm of $ P_{t+1} $ as a function of the norm of $ P_t $. However, an example can be constructed where $P_{t+1}$ as a function of $P_t$ becomes unbounded as $P_t$ approaches the boundary of the set positive-definite matrices that make $K_{t+1}$ unstable. This does not happen in the scalar case since this boundary is smaller than $\tilde{P}_t$, the minimum achievable $P_t$. Figure~\ref{Example1} depicts the norm of $P_{t+1}$ for different trials of selecting $P_t$. 
For each trial, the $P_t$ 
is chosen as $P_t = P^* + \Omega$, where $P^*$ is the minimum achievable $P_t$ for a stable matrix $K_t$, and $\Omega$ is a positive definite matrix. It can be seen that the norm of $P_{t+1}$ for some trials gets very large. For example, for $P_t$
\begin{align*}
P_t=
\begin{pmatrix*}[r]
	18714& \quad -312& \quad 291& \\
   -312& \quad 82149& \quad -144& \\
    291& \quad -144& \quad 14220& 
\end{pmatrix*},
\end{align*}
the matrix $A-BK_{t+1}$ has the eigenvalues
\begin{align*}
\lambda(A-BK_{t+1})=
\begin{pmatrix}
  -0.999996 \\
   0.002971 \\
  -0.000047
\end{pmatrix},
\end{align*}
and the first eigenvalue that is near $1$, which makes the norm of $P_{t+1}$ around the order of $7.7\times 10^8$.
However, in several simulations of online Riccati algorithm, we observed that changes in $P_t$ as a result of changes in bounded $\bar{Q}_t$ and $\bar{R}$ do not make $K_{t+1}$ to get close to the unstable policy boundary, and hence $P_{t+1}$ cannot get unbounded. We will show this behaviour in the following experiment.}
\end{remarkappendix}
\begin{exampleappendix}\label{example:A}{\em
In order to observe the behaviour of matrices $P_t$ over time, a linear discrete-time control system with $n=7$ states and $m=5$ control actions is considered, where the matrices $(A,B)$ are fixed. 
\begin{figure*}[htb!]
  \centering
    {
\includegraphics[scale=0.6]{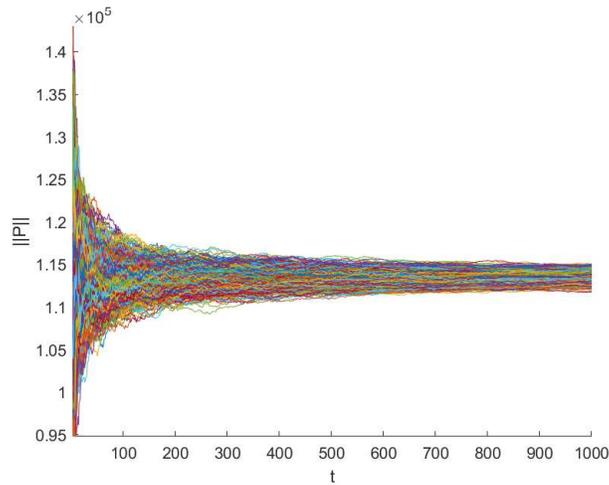}
\caption{The norm of $P_t$ over time for 1000 trials is shown. For each trial, a sequence of matrices $Q_t$ and $R_t$ with Wishart distribution is generated and the sequence $P_t$ is generated using the online Riccati algorithm.}
\label{Example2}
}
\end{figure*}
We used several trials, where for each trial a sequence of positive definite random matrices $Q_t$ and $R_t$ with Wishart distribution is generated and we used the online Riccati algorithm with different initialization $K_1$ to generate the sequence $P_t$. Figure~\ref{Example1} shows the graph of the norm of $P_t$ over time for each trial. Clearly, $P_t$ stays bounded. Similar property is observed in all our simulation studies. Understanding why this boundedness occurs and if this is generally true is an important open problem, and appears to be difficult in light of the previous remark.}
\end{exampleappendix}

\end{document}